\documentclass[11pt]{amsart}
\usepackage{verbatim,amssymb,amsmath,graphicx,hyperref}

  \scrollmode

\newcommand{\R}{{\mathbb R}}
\newcommand{\N}{{\mathbb N}}

\newcommand{\EE}{{\mathbb E}}
\newcommand{\PP}{{\mathbb P}}

\newcommand{\eul}{{\widehat X}^{\delta}}
\newcommand{\eultr}{{\widehat Z}^{\delta}}
\newcommand{\ind}{1}

\newcommand{\usn}{\underline {s}^{\delta}}
\newcommand{\utn}{\underline {t}^{\delta}}

\newcommand{\uun}{\underline {u}^{\delta}}

\newcommand{\sgn}{\operatorname{sgn}}
\newcommand{\eps}{\varepsilon}

\newcommand{\Tm}{{\mathcal T}}
\theoremstyle{plain}
\newtheorem{theorem}{Theorem}
\newtheorem{prop}{Proposition}
\newtheorem{lemma}{Lemma}

\theoremstyle{definition}

\begin{document}
\title[An adaptive method of order 1  for SDEs with discontinuous drift coefficient]{An adaptive strong order 1 method for SDEs with discontinuous drift coefficient}


\author[Yaroslavtseva]
{Larisa Yaroslavtseva}
\address{
Faculty of Computer Science and Mathematics \\
University of Passau\\
Innstrasse 33\\
94032 Passau\\
Germany} \email{larisa.yaroslavtseva@uni-passau.de}

\begin{abstract}
In recent years, an intensive study of strong approximation of stochastic differential equations (SDEs) with a drift coefficient that may have discontinuities in space has begun. In many of these results it is assumed that the drift coefficient satisfies piecewise regularity conditions and the diffusion coefficient is Lipschitz continuous  and non-degenerate at the discontinuity points of the drift coefficient.
For scalar SDEs of that type the best $L_p$-error rate known so far  for approximation of the solution at the final time point  is $3/4$ in terms of the number of evaluations of the driving Brownian motion and it is achieved by the transformed equidistant quasi-Milstein scheme, see~\cite{MGY19b}. Recently in~\cite{MGY21} it has been  shown  that for such SDEs the $L_p$-error rate  $3/4$ can not be improved in general  by no numerical method based on  evaluations of the driving Brownian motion at fixed time points.  In the present article we  construct for the first time in the literature
 a   method based on sequential evaluations of the driving Brownian motion, which achieves an $L_p$-error rate of at least $1$ in terms of the average number of evaluations of the driving Brownian motion for such SDEs.
\end{abstract}

\maketitle

\section{Introduction}

In this article we consider a scalar autonomous stochastic differential equation (SDE)
\begin{equation}\label{sde000}
\begin{aligned}
dX_t & = \mu(X_t) \, dt + \sigma(X_t) \, dW_t, \quad t\geq 0,\\
X_0 & = x_0,
\end{aligned}
\end{equation}
 where $x_0\in\R$ is the initial value,  $\mu\colon\R\to\R$ is the drift coefficient, $\sigma\colon \R\to\R$ is the diffusion coefficient, $W=(W_t)_{t\geq 0}$ is a  $1$-dimensional Brownian motion  and we assume that the SDE~\eqref{sde000} has a unique strong solution $X$. Our computational task is  $L_p$-approximation of $X_1$  by numerical methods that are based on finitely many evaluations of the driving Brownian motion $W$ at  points in $[0,1]$ in the case when the drift coefficient $\mu$ may have finitely many discontinuity points.

Strong approximation of  SDEs with a discontinuous drift coefficient has gained a lot of interest in the literature in recent years.  See~\cite{ g98b, gk96b} for results on convergence in probability and  almost sure convergence  of the Euler-Maruyama scheme  and~\cite{DG18, GLN17, HalidiasKloeden2008,  LS16,  LS15b, LS18, MGY19b, MGY20, NS19, NSS19,  Tag16, Tag2017b, Tag2017a, PS19}
for results on $L_p$-approximation. In many of these articles it is assumed that the drift coefficient satisfies piecewise regularity conditions and the diffusion coefficient is Lipschitz continuous  and non-degenerate at the discontinuity points of the drift coefficient. For  SDEs of that type the best $L_p$-error rate known up to now for approximation of $X_1$    is $3/4$, see~\cite{MGY19b}.
 In the present article we  construct for the first time in the literature
 a numerical  method, which achieves an $L_p$-error rate of at least $1$ for  such SDEs. 
 
To be more precise, let us consider the following assumptions on the coefficients $\mu$ and $\sigma$. 

\begin{itemize}
\item[($\mu$1)] There exist $k\in\N$ 
and $\xi_0, \ldots, \xi_{k+1}\in [-\infty,\infty]$ with $-\infty=\xi_0<\xi_1<\ldots < \xi_k <\xi_{k+1}=\infty$ such that
 $\mu$ is Lipschitz continuous on the interval $(\xi_{i-1}, \xi_i)$ for all $i\in\{1, \ldots, k+1\}$,
\item[($\sigma$1)] $\sigma$ is Lipschitz continuous on $\R$ and $\sigma(\xi_i) \neq 0$ for all $i\in\{1,\ldots,k\}$,
\end{itemize}

If  ($\mu$1) and ($\sigma$1) hold then   
the SDE~\eqref{sde000} has a unique strong solution, see~\cite{LS16}.   In~\cite{ LS16,  LS15b, LS18, MGY19b, MGY20,  NSS19} the $L_p$-approximation of $X_1$ under the assumptions
 ($\mu$1) and ($\sigma$1)
has been analyzed.
In particular, in~\cite{LS16, LS15b} the transformed equidistant  Euler-Maruyama 
scheme 
has been  
constructed, which  achieves an $L_2$-error rate of at least $1/2$  in terms of the number of evaluations of the driving Brownian motinon $W$.   
After that,
 in~\cite{NSS19} an adaptive Euler-Maruyama scheme
has been
constructed, which achieves up to a logarithmic factor
an $L_2$-error rate of at least $1/2$ in terms of the average number of evaluations of $W$ used by the scheme. Finally, in~\cite{MGY20} it
has been 
proven that the classical equidistant  Euler-Maruyama 
scheme 
achieves for all $p\in [1,\infty)$ an $L_p$-error rate of at least $1/2$  
in terms of the number  of evaluations of $W$
as in the case of SDEs with 
globally Lipschitz continuous coefficients.

In~\cite{MGY19b} the first higher-order method has been constructed for such SDEs. This method is based on  equidistant evaluations  of  $W$ and 
 achieves for all $p\in[1, \infty)$ an $L_p$-error rate of at least $3/4$ in terms of the number of evaluations of $W$ if $\mu$ and $\sigma$ satisfy ($\mu$1) and ($\sigma$1) and additionally the following  piecewise regularity assumptions \vspace{0.1cm}
\begin{itemize}
\item[($\mu$2)]  $\mu$ has a Lipschitz continuous derivative on   $(\xi_{i-1}, \xi_i)$  for every $i\in\{1, \ldots, k+1\}$,\vspace{0.1cm}
\item[($\sigma$2)]  $\sigma$ has a Lipschitz continuous derivative on  $(\xi_{i-1}, \xi_i)$  for every $i\in\{1, \ldots, k+1\}$. \vspace{0.1cm}
\end{itemize}
Furthermore, in~\cite{NS19} it has been shown that for SDEs~\eqref{sde000} with additive noise and a bounded and piecewise $C^2_b$ drift coefficient $\mu$ the equidistant Euler-Maruyama scheme in fact achieves an $L_2$-error rate of  at least $3/4-$ in terms of the number of evaluations of $W$. Note that in this case the Euler-Maruayama scheme coincides with the Milstein scheme.

Recently in~\cite{MGY21} it has been shown  that an $L_p$-error rate better than $3/4$ can not be achieved in general under the assumptions ($\mu$1), ($\mu$2), ($\sigma$1) and ($\sigma$2) by no numerical method based on  evaluations of $W$ at fixed time points in $[0,1]$. More precisely,  it has been proven in~\cite{MGY21} that if
 $\sigma=1$ and if  $\mu$ satisfies ($\mu$1) and ($\mu$2), $\mu$ is bounded, increasing and there exists $i\in\{1,\dots,k\}$ such that $\mu(\xi_i+)\not=\mu(\xi_i-)$, then there exists $c\in(0, \infty)$
such that   for all $p\in[1, \infty)$ and  all  $n\in\N$, 
\begin{equation}\label{lb3}
\inf_{\substack{
       t_1,\dots ,t_n \in [0,1]\\
        g \colon \R^n \to \R \text{ measurable} \\
       }}	 \EE\bigl[|X_1-g(W_{t_1}, \ldots, W_{t_n})|^p\bigr]^{1/p}\geq \frac{c}{n^{3/4}}.
\end{equation}

Note that the lower bound \eqref{lb3} does not cover adaptive methods, i.e. methods that may choose the number as well as the
location of the evaluations of the Brownian motion  $W$ in a sequential
 way dependent on the values of $W$ observed so far. See e.g.~\cite{Gaines1997, Hoel2012,Hoel2014,  LambaMattinglyStuart2007, MG02_habil,m04, NSS19,  RW2006} for examples of such methods.  It is well-known  
 that for a large class of SDEs~\eqref{sde000} with  globally Lipschitz continuous coefficients  the best possible $L_p$-error rate  that can be achieved by non-adaptive methods coincides with the best possible $L_p$-error rate  that can be achieved by adaptive methods  and is equal to $1$, see~\cite{MG02_habil,m04}.  Moreover, up to now there is no example of an SDE with globally Lipschitz continuous coefficients known in the literature, for which adaptive methods are superior to non-adaptive ones with respect to the $L_p$-error rate.
However, the superiority of adaptive methods to non-adaptive ones with respect to the $L_p$-error rate has recently been demonstrated  in~\cite{HH16, MGRY2018} for some examples of SDEs with non-globally Lipschitz continuous drift or diffusion coefficients. 

In view of the latter results it is 
 natural to ask    whether 
 there exists an adaptive method 
 that achieves under the assumptions ($\mu$1), ($\mu$2) and ($\sigma$1), ($\sigma$2) a better $L_p$-error rate than the rate $3/4$.
  To the best of our knowledge the answer to this question
was not known  in the literature up to now. In the present article we answer this question in the positive.  More precisely, we construct  a family of approximations $\eul_1$ with $\delta\in(0, \delta_0]$ for some $\delta_0>0$ such that each approximation $\eul_1$ is based on at most  $c\cdot \delta^{-1}$ adaptively chosen evaluations of $W$ in the interval $[0,1]$ on average and such that  for all $p\in[1, \infty)$ and all $\delta\in(0, \delta_0]$,
\[
\EE\bigl[|X_1- \eul_1|^p\bigr]^{1/p} \le c(p)\cdot \delta,
\]
where the constants $c, c(p)\in (0, \infty)$  do not depend on $\delta$, see Theorem \ref{Thm2}. Thus, the  approximations $\eul_1$ achieve an $L_p$-error rate  of at least $1$ in terms of the average number of evaluations of $W$.
The methods $\eul_1$ are obtained by applying a suitable transformation $G\colon\R\to\R$ to the strong solution $X$ of the SDE \eqref{sde000} such that the transformed solution $Z=(G(X_t))_{t\geq 0}$ is a strong solution of a new SDE with  coefficients $\widetilde \mu$ and $\widetilde \sigma$ which  satisfy ($\mu$1), ($\mu$2) and ($\sigma$1), ($\sigma$2), respectively, and such that  $\widetilde \mu$ is continuous, which implies that $\widetilde\mu$ is Lipschitz continuous. An adaptive  quasi-Milstein scheme $\widehat Z^\delta=(\widehat Z^\delta_t)_{t\geq 0}$ is  used to approximate $Z$ and the approximation $\eul_1$ is then given by $G^{-1}(\widehat Z^\delta_1)$. The adaptive time stepping   strategy used for the adaptive  quasi-Milstein scheme $\widehat Z^\delta$ is an appropriate modification of the adaptive time stepping     strategy used for the adaptive Euler-Maruyama scheme in~\cite{NSS19}. 
We add that an $L_p$-error rate better than $1$ can not be achieved in general under the assumptions ($\mu$1), ($\mu$2) and ($\sigma$1), ($\sigma$2) by no adaptive method based on finitely many evaluations of $W$, see~\cite{hhmg2019, MG02_habil,m04} for corresponding lower error bounds.

The implementation of our method requires the ability to evaluate  the functions $G$ and $G^{-1}$ at each step of the adaptive  quasi-Milstein scheme $\widehat Z^\delta$. While the transformation $G$ is known explicitly, this is so far  not the case for $G^{-1}$, and therefore a numerical inverse of $G$ has to be used to approximate $G^{-1}$.  This makes our method rather slow in practice. 
We conjecture however that the transformation of the SDE~\eqref{sde000} is actually not needed and that an adaptive quasi-Milstein scheme  for the SDE~\eqref{sde000} itself achieves under the assumptions ($\mu$1), ($\mu$2) and ($\sigma$1), ($\sigma$2) an $L_p$-error rate of at least $1$ in terms of the average number of evaluations of $W$. The proof of this conjecture will be the subject of future work.

We briefly describe the content of the paper. In Section~\ref{Not} we introduce some notation. Section~\ref{QM} contains the construction and the error and cost analysis 
of the adaptive quasi-Milstein scheme in the case when the coefficients of the SDE~\eqref{sde000} satisfy the assumptions ($\mu$1), ($\mu$2) and ($\sigma$1), ($\sigma$2) and the drift coefficient is continuous, see Theorem~\ref{Thm1}. In Section~\ref{threefour} we 
 introduce
the bi-Lipschitz transformation $G$ that is then used to construct a method of order $1$ under the assumptions ($\mu$1), ($\mu$2) and ($\sigma$1), ($\sigma$2),  see Theorem~\ref{Thm2}.
Section~\ref{Proofs} is devoted to the proof of Theorem~\ref{Thm1}. 

\section{Notation}\label{Not}
For  $A\subset \R$ and $x\in\R$ we put $d(x, A)=\inf\{|x-y|\colon y\in A\}$. For a function $f\colon \R\to\R$ we define $d_f\colon \R\to\R$ by
\[
d_f(x) = \begin{cases} f'(x), & \text{if $f$ is differentiable in $x$},\\
0, & \text{otherwise.}\end{cases}
\]

\section{An adaptive quasi-Milstein scheme for SDEs with Lipschitz continuous coefficients}\label{QM}

Let
$ ( \Omega, \mathcal{F}, \PP ) $ 
be a complete probability space,
let
$
  W \colon [0,\infty) \times \Omega \to \R
$
be a Brownian motion
on $ ( \Omega, \mathcal{F}, \PP )$, let $x_0\in\R$ and let $\mu\colon\R\to\R$ and $ \sigma\colon\R\to\R$ be functions that satisfy the assumptions ($\mu$1), ($\mu$2)  and ($\sigma$1), ($\sigma$2), respectively, and assume that $\mu$ is continuous.
We consider the SDE
\begin{equation}\label{sde01}
\begin{aligned}
dX_t & = \mu(X_t) \, dt + \sigma(X_t) \, dW_t, \quad t\geq 0,\\
X_0 & = x_0.
\end{aligned}
\end{equation}
Observe that in this case both $\mu$ and $\sigma$ are Lipschitz continuous on $\R$, and therefore the SDE \eqref{sde01}  has a unique strong solution 
and for every $p\in [1,\infty)$ it holds
\begin{equation}\label{mom}
\EE\bigl[\sup_{t\in[0,1]}|X_t|^p\bigr] < \infty.
\end{equation}

Put $\Theta=\{\xi_1, \ldots, \xi_k\}$
and for $\varepsilon >0$ let
\[
\Theta^{\varepsilon}=\{x\in\R\colon d(x, \Theta)<\varepsilon\}.
\]
Let $\varepsilon_0\in(0,1]$ and assume that
\[
\varepsilon_0\leq \frac{1}{2}\min\{\xi_i-\xi_{i-1}\colon i=2, \ldots, k\}
\]
if $k\geq 2$.
For $\delta>0$ put
\[
\varepsilon_1^{\delta}=\sqrt \delta\cdot\log^2(1/\delta), \qquad \varepsilon_2^{\delta}=\delta\cdot \log^4(1/\delta).
\]
Let $\delta_0\in(0,1)$ be small enought such that for all $\delta\in (0, \delta_0]$ it holds
\begin{equation}\label{eps}
\varepsilon_2^{\delta}\leq\varepsilon_1^{\delta}\leq \varepsilon_0/2.
\end{equation}
For $\delta\in (0, \delta_0]$ we define a time-continuous adaptive  quasi-Milstein scheme 
$\eul=(\eul_t)_{t\geq 0}$  recursively by
\begin{equation}\label{mil1}
\tau_0^{\delta}=0, \quad \eul_{\tau_0^{\delta}}=x_0
\end{equation}
and
\begin{equation}\label{mil2}
\begin{aligned}
\tau_{i+1}^{\delta}&= \tau_i^{\delta}+h^{\delta}(\eul_{\tau_i^{\delta}}),\\
\eul_{t}&=\eul_{\tau_i^{\delta}}+\mu(\eul_{\tau_i^{\delta}})\cdot (t-\tau_i^{\delta})+\sigma(\eul_{\tau_i^{\delta}})\cdot (W_t-W_{\tau_i^{\delta}})\\
&\qquad\qquad+\frac{1}{2}\sigma d_\sigma (\eul_{\tau_i^{\delta}})\cdot\bigl((W_t-W_{\tau_i^{\delta}})^2-(t-\tau_i^{\delta})\bigr), \quad t\in (\tau_i^{\delta},\tau_{i+1}^{\delta}],
\end{aligned}
\end{equation}
for $i\in\N_0$,
where the step size function $h^{\delta}\colon\R\to (0,1)$ is defined by
\begin{align}\label{step}
h^{\delta}(x) = \begin{cases} 
\delta, & x\not\in\Theta^{\varepsilon_1^{\delta}},\\
\Bigl(\frac{d(x, \Theta)}{\log^2(1/\delta)}\Bigr)^2, & x\in \Theta^{\varepsilon_1^{\delta}}\setminus \Theta^{\varepsilon_2^{\delta}},\\
\delta^2\cdot \log^4(1/\delta), & x\in \Theta^{\varepsilon_2^{\delta}}.
\end{cases}
\end{align}
Note that the assumption \eqref{eps} implies that  $\Theta^{\varepsilon_2^{\delta}}\subseteq \Theta^{\varepsilon_1^{\delta}}$ for all $\delta\in (0, \delta_0]$ and hence $h^{\delta}$ is well-defined  for all $\delta\in (0, \delta_0]$. Moreover,  $h^\delta$ is continuous and it holds
\begin{equation}\label{vv2}
\delta^2\cdot \log^4(1/\delta)\leq h^{\delta}\leq \delta
\end{equation}  
for all $\delta\in (0, \delta_0]$. We add that the step size function $h^{\delta}$ we use for the adaptive quasi-Milstein scheme 
$\eul$  is an appropriate modification of the step size function used for the adaptive Euler-Maruyama scheme in~\cite{NSS19}. 

For $\delta\in (0, \delta_0]$ let $N(\eul_1)$ denote the number of evaluations of $W$ used to compute $\eul_1$, i.e. 
\[
N(\eul_1)=\min\{i\in\N\colon \tau_i^{\delta}\geq 1\}.
\]
Clearly, for all $\delta\in (0, \delta_0]$,
\[
N(\eul_1)\leq \lceil \delta^{-2}\log^{-4}(1/\delta)\rceil.
\]

 We have the following  upper bounds for  the $p$-th root of
the $p$-th mean of the maximum error of  $\eul$ on the time interval $[0,1]$
 and for the average number of evaluations of $W$ used to compute  $\eul_1$.
\begin{theorem}\label{Thm1}
Assume  ($\mu$1), ($\mu$2)  and ($\sigma$1), ($\sigma$2) and assume that $\mu$ is continuous. Let  $p\in [1,\infty)$. Then
there exists $c_1, c_2\in(0, \infty)$ such that for all $\delta\in(0,\delta_0]$, 
\begin{equation}\label{ll33}
\EE\bigl[\sup_{t\in[0,1]}|X_t-\eul_t|^p\bigr]^{1/p}\leq c_1\cdot \delta 
\end{equation}
and 
\begin{equation}\label{ll32}
\EE [N(\eul_1)]\leq c_2\cdot \delta^{-1}.
\end{equation}
\end{theorem}
The proof of Theorem \ref{Thm1} is postponed to Section~\ref{Proofs}.

\section{An adaptive strong order 1 method for SDEs with discontinuous drift coefficient}\label{threefour}
As in Section~\ref{QM} we consider a complete probability space
$ ( \Omega, \mathcal{F}, \PP ) $  and we assume that 
$
  W \colon [0,\infty) \times \Omega \to \R
$
is a Brownian motion
on $ ( \Omega, \mathcal{F}, \PP )$. In contrast to Section~\ref{QM} we now turn to SDEs with a drift coefficient $\mu$ that may have discontinuity points.

Let $x_0\in\R$ and let $\mu\colon\R\to\R$ and $ \sigma\colon\R\to\R$ be functions that satisfy the assumptions ($\mu$1), ($\mu$2)  and ($\sigma$1), ($\sigma$2), respectively. 
For later purposes we note that ($\mu$1) implies the existence of the one-sided limits $\mu(\xi_i-)$ and $\mu(\xi_i+)$ for all $i\in\{1,\dots,k\}$.
We consider the SDE
\begin{equation}\label{sde0}
\begin{aligned}
dX_t & = \mu(X_t) \, dt + \sigma(X_t) \, dW_t, \quad t\geq 0,\\
X_0 & = x_0,
\end{aligned}
\end{equation}
which has a unique strong solution,
see~\cite[Theorem 2.2]{LS16}.

We now constuct an adaptive method for approximating the strong solution  of the SDE \eqref{sde0} at the  time $1$.
To this end we employ
the
transformation strategy from~\cite{MGY19b}. We use that 
$X_1$ can be obtained by applying a Lipschitz continuous transformation to the strong solution of an SDE with coefficients $\widetilde\mu, \widetilde\sigma$ satisfying the assumptions ($\mu$1), ($\mu$2)  and ($\sigma$1), ($\sigma$2), respectively, such that $\widetilde\mu$ is continuous, and then we employ Theorem~\ref{Thm1}.

We start by introducing the transformation procedure from~\cite{MGY19b}. For $k\in\N$, 
\[
z\in\Tm_k=\{(z_1,\dots,z_k)\in\R^k\colon z_1<\dots<z_k\}
\]
 and $\alpha=(\alpha_1,\dots,\alpha_k)\in\R^k$ we put
\[
\rho_{z,\alpha} =  \begin{cases}
\frac{1}{8 |\alpha_1|}, & \text{if }k=1, \\
\min\bigl(\bigl\{\frac{1}{8 |\alpha_i|}\colon i\in \{1, \ldots, k\}\bigr\} \cup \bigl\{ \frac{z_i-z_{i-1}}{2}\colon i\in \{2, \ldots, k\}\bigr\} \bigr),& \text{if }k\geq 2,
\end{cases}
\]
where we use the convention $1/0 =\infty$. Let $\phi\colon\R\to\R$ be given by
\begin{equation}\label{phi}
\phi(x)=(1-x^2)^4\cdot \ind_{[-1, 1]}(x). 
\end{equation}
For all $k\in\N$, $z\in \Tm_k$, $\alpha\in\R^k$ and $\nu\in (0,\rho_{z,\alpha})$ we define a function $G_{z,\alpha,\nu}\colon\R\to\R$ by
\begin{equation}\label{fct1}
G_{z,\alpha,\nu}(x) = x+\sum_{i=1}^k \alpha_i\cdot (x-z_i)\cdot |x-z_i|\cdot \phi \Bigl(\frac{x-z_i}{\nu}\Bigr).
\end{equation}

The following two technical lemmas provide the properties of the mappings $G_{z,\alpha,\nu}$ that are crucial for our purposes. For the proofs of both lemmas see~\cite{MGY19b}.

\begin{lemma}\label{lemx1}
Let $k\in\N$, $z\in \Tm_k$, $\alpha\in\R^k$, $\nu\in (0,\rho_{z,\alpha})$ and put $z_0=-\infty$ and $z_{k+1}= \infty$. The function $G_{z,\alpha,\nu}$ has the following properties.
\begin{itemize}
\item[(i)] $G_{z,\alpha,\nu}$  is differentiable on $\R$ with a Lipschitz continuous derivative $G'_{z,\alpha,\nu}$ that satisfies $\inf_{x\in\R} G_{z,\alpha,\nu}'(x)>0$.
In particular, $G_{z,\alpha,\nu}$ has an inverse $G_{z,\alpha,\nu}^{-1}\colon \R\to \R$ that is Lipschitz continuous. 
\item[(ii)] For every $i\in\{1,\dots,k+1\}$, the function $G'_{z,\alpha,\nu}$ is  differentiable on $(z_{i-1},z_i)$ with Lipschitz continuous derivatives $G''_{z,\alpha,\nu}$.
\item[(iii)] For every $i\in\{1,\dots,k\}$ the one-sided limits  $G''_{z,\alpha,\nu}(z_i-) $ and $G''_{z,\alpha,\nu}(z_i+)$ exist and satisfy
\[
G''_{z,\alpha,\nu}(z_i-) = -2\alpha_i,\quad G''_{z,\alpha,\nu}(z_i+) = 2\alpha_i.
\]
\end{itemize}
\end{lemma}

\begin{lemma}\label{transform1} Assume ($\mu$1), ($\mu$2)  and ($\sigma$1), ($\sigma$2). Put $\xi=(\xi_1,\dots,\xi_k)$, define $\alpha=(\alpha_1,\dots,\alpha_k)\in \R^k$ by
\[
\alpha_i =\frac{\mu(\xi_i-)-\mu(\xi_i+)}{2 \sigma^2(\xi_i)}
\]
 for  $i\in\{1,\dots,k\}$, and let $\nu\in (0,\rho_{\xi,\alpha})$. Consider the function $G_{\xi,\alpha,\nu}$ and extend $G''_{\xi,\alpha,\nu}\colon \cup_{i=1}^{k+1} (\xi_{i-1},\xi_i)\to \R$ to the whole real line by taking
 \[
 G''_{\xi,\alpha,\nu}(\xi_i) = 2\alpha_i + 2\,\frac{\mu(\xi_i+)-\mu(\xi_i)}{\sigma^2(\xi_i)}  
\]
for $i\in\{1, \ldots, k\}$. Then the functions
\begin{equation}\label{tildecoeff}
\widetilde \mu=(G_{\xi,\alpha,\nu}'\cdot \mu+\tfrac{1}{2}G_{\xi,\alpha,\nu}''\cdot\sigma^2)\circ G_{\xi,\alpha,\nu}^{-1} \, \text{ and }\, \widetilde\sigma=(G_{\xi,\alpha,\nu}'\cdot\sigma)\circ G_{\xi,\alpha,\nu}^{-1}
\end{equation}
satisfy the assumptions ($\mu$1), ($\mu$2)  and ($\sigma$1), ($\sigma$2), respectively,  and $\widetilde\mu$ is continuous. 
\end{lemma}

We turn to the transformation of the SDE~\eqref{sde0}. Take $\xi,\alpha,\nu$ as in Lemma~\ref{transform1} and define a stochastic process $Z\colon [0,\infty)\times \Omega\to \R$ by  
\begin{equation}\label{tr}
Z_t = G_{\xi,\alpha,\nu}(X_t),\quad t\geq 0.
\end{equation}
Then the process $Z$  is the unique strong solution of the SDE
\begin{equation}\label{sde1}
\begin{aligned}
dZ_t & = \widetilde\mu(Z_t) \, dt + \widetilde\sigma(Z_t) \, dW_t, \quad t\geq 0,\\
Z_0 & = G_{\xi,\alpha,\nu}(x_0)
\end{aligned}
\end{equation}
with $\widetilde \mu$ and $\widetilde \sigma$ given by~\eqref{tildecoeff}, see~\cite{MGY19b}.
For every $\delta\in(0,\delta_0]$ we use $\eultr=(\eultr_t)_{t\geq 0}$ to denote the 
    time-continuous adaptive
quasi-Milstein  scheme \eqref{mil1}, \eqref{mil2} associated to the SDE \eqref{sde1}, i.e. $\eultr$ is defined  recursively by
\begin{equation}\label{m1}
\tau_0^{\delta}=0, \quad \eultr_{\tau_0^{\delta}}=G_{\xi,\alpha,\nu}(x_0)
\end{equation}
and
\begin{equation}\label{m2}
\begin{aligned}
\tau_{i+1}^{\delta}&=\tau_i^{\delta}+h^{\delta}(\eultr_{\tau_i^{\delta}}),\\
\eultr_{t}&=\eultr_{\tau_i^{\delta}}+\widetilde\mu(\eultr_{\tau_i^{\delta}})\cdot (t-\tau_i^{\delta})+\widetilde\sigma(\eultr_{\tau_i^{\delta}})\cdot (W_t-W_{\tau_i^{\delta}})\\
&\qquad\qquad+\frac{1}{2}\widetilde\sigma d_{\widetilde\sigma} (\eultr_{\tau_i^{\delta}})\cdot\bigl((W_t-W_{\tau_i^{\delta}})^2-(t-\tau_i^{\delta})\bigr), \quad t\in (\tau_i^{\delta},\tau_{i+1}^{\delta}],
\end{aligned}
\end{equation}
for $i\in\N_0$, where the step size function $h^{\delta}$ is given by \eqref{step}.

 We approximate $X$ by the stochastic process  $\eul=(\eul_{t})_{t\geq 0}$ with  $\eul_t= G_{\xi,\alpha,\nu}^{-1}(\eultr_{t})$, $t\geq 0$. For  $\delta\in (0, \delta_0]$ let $N(\eul_1)$ denote the number of evaluations of $W$ used to compute  $\eul_1$. We have the following  upper bounds for  the $p$-th root of
the $p$-th mean of the maximum error of  $\eul$ on the time interval $[0,1]$
 and for the average number of evaluations of $W$ used to compute  $\eul_1$.

\begin{theorem}\label{Thm2} 
Assume ($\mu$1), ($\mu$2)  and ($\sigma$1), ($\sigma$2). Let  $p\in [1,\infty)$. Then
there exists $c_1, c_2\in(0, \infty)$ such that for all $\delta\in(0,\delta_0]$, 
\begin{equation}\label{Ll33}
\EE\bigl[\sup_{t\in[0,1]}|X_t-\eul_t|^p\bigr]^{1/p}\leq c_1\cdot \delta 
\end{equation}
and 
\begin{equation}\label{Ll32}
\EE [N(\eul_1)]\leq c_2\cdot \delta^{-1}.
\end{equation}
\end{theorem}
\begin{proof}
Using the Lipschitz continuity of $G_{\xi,\alpha,\nu}^{-1}$, 
see Lemma~\ref{lemx1}(i),
 the fact that $\widetilde \mu$ and $\widetilde \sigma$ satisfy the assumptions ($\mu$1), ($\mu$2)  and ($\sigma$1), ($\sigma$2), respectively, and that $\widetilde \mu$ is continuous as well as the estimate \eqref{ll33} in Theorem~\ref{Thm1} we obtain that there exist $c_1, c_2\in (0, \infty)$ such that for all $\delta\in(0,\delta_0]$,
\begin{align*}
\EE\bigl[\sup_{t\in[0,1]}|X_t-\eul_t|^p\bigr]^{1/p}=\EE\bigl[\sup_{t\in[0,1]}|X_t-G_{\xi,\alpha,\nu}^{-1}(\eultr_t)|^p\bigr]^{1/p}\leq c_1\cdot \EE\bigl[\sup_{t\in[0,1]}|Z_t- \eultr_t|^p\bigr]^{1/p}\leq c_2\cdot\delta.
\end{align*}
Thus, \eqref{Ll33} holds. The estimate \eqref{Ll32} follows from the fact that $N(\eul_1)=N(\eultr_1)$ and the estimate \eqref{ll32} in Theorem~\ref{Thm1}.
\end{proof}

\section{Proof of  Theorem~\ref{Thm1}}\label{Proofs}
Throughout this section we 
assume that $\mu$ and $\sigma$ satisfy 
($\mu$1), ($\mu$2)  and ($\sigma$1), ($\sigma$2), respectively, and that $\mu$ is continuous. Moreover, for  $\delta\in(0,\delta_0]$ and $t\in [0,1]$ we put 
\[
\utn = \max\{\tau_i^{\delta}\colon i\in \N_0,  \tau_i^{\delta}\leq t\}.
\] 
We first briefly describe the structure of the proof of the error estimate \eqref{ll33} in Theorem~\ref{Thm1} and the relation of our analysis and the error analysis of the equidistant quasi-Milstein scheme in~\cite{MGY19b}.  Let ${\widehat X}^{\delta, eq}=({\widehat X}^{\delta, eq}_t)_{t\geq 0}$ denote the equidistant quasi-Milstein scheme with step size $\delta$, i.e. ${\widehat X}^{\delta, eq}$ is defined in the same way as $\eul$ in \eqref{mil2}, but with $h^\delta=\delta$ in place of \eqref{step}. 
For simplicity let us restrict to the case $p=2$.  In~\cite{MGY19b} it is shown that there exists $c\in(0, \infty)$ such that for all $\delta\in\{1/n\colon n\in\N\}$,
\begin{equation}\label{sk1}
\EE\bigl[\sup_{t\in[0,1]}|X_t-{\widehat X}^{\delta, eq}_t|^2\bigr]^{1/2}
 \le c\cdot \delta + c\cdot \Bigl(\int_0^1 \EE\bigl[ |{\widehat X}^{\delta, eq}_t-{\widehat X}^{\delta, eq}_{\utn}|^2\cdot \ind_S ({\widehat X}^{\delta, eq}_t,{\widehat X}^{\delta, eq}_{\utn})\bigr]\, dt\Bigr)^{1/2},
\end{equation}
where 
\[
S= \Bigl(\bigcup_{i=1}^{k+1} (\xi_{i-1},\xi_i)^2\Bigr)^c
\]
is the set of pairs
$(x,y)$ in $\R^2$, which do not allow 
for a joint Lipschitz estimate of $ |d_\mu(x) -d_\mu(y)|$ or of $ |d_\sigma(x) -d_\sigma(y)|$ if $\mu$ or $\sigma$ is not differentiable at one of the points $\xi_1, \ldots, \xi_k$. Transforming the condition $({\widehat X}^{\delta, eq}_t,{\widehat X}^{\delta, eq}_{\utn})\in S$  into a condition solely on the sizes of  the random variables $|\eul_{\utn-(t-\utn)} - \xi_i|$,
$|\eul_{\utn-(t-\utn)} -\eul_{\utn}|$ and $|\eul_{\utn}-\eul_{t}|$, where $\xi_i$ lies between  ${\widehat X}^{\delta, eq}_t$ and ${\widehat X}^{\delta, eq}_{\utn}$, and employing a Markov-type property of  ${\widehat X}^{\delta, eq}$ 
and  occupation time estimates for ${\widehat X}^{\delta, eq}$  it is  shown in~\cite{MGY19b} that
there exists  $c\in (0,\infty)$ such that for all $\delta\in\{1/n\colon n\in\N\}$,
\begin{equation}\label{sk3}
\int_0^1 \EE\bigl[ |{\widehat X}^{\delta, eq}_t-{\widehat X}^{\delta, eq}_{\utn}|^2\cdot \ind_S ({\widehat X}^{\delta, eq}_t,{\widehat X}^{\delta, eq}_{\utn})\bigr]\, dt\leq c\cdot\delta^{3/2}.
\end{equation}
Combining \eqref{sk1} and~\eqref{sk3} yields
the rate of convergence $3/4$ 
for the root mean square of the maximum error of the equidistant quasi-Milstein scheme ${\widehat X}^{\delta, eq}$ on the time interval $[0,1]$.

Our proof of \eqref{ll33} reproduces the estimate~\eqref{sk1}.  
Proceeding similarly to~\cite[Subsection 5.3]{MGY19b} we show that 
there exists $c\in(0, \infty)$ such that for all $\delta\in(0, \delta_0]$ the adaptive quasi-Milstein scheme $\eul$ satisfies
\begin{equation}\label{sk4}
\EE\bigl[\sup_{t\in[0,1]}|X_t-\eul_t|^2\bigr]^{1/2}
 \le c\cdot \delta + c\cdot \Bigl(\int_0^1 \EE\bigl[|\eul_t-\eul_{\utn}|^2\cdot \ind_S (\eul_t,\eul_{\utn})\bigr]\, dt\Bigr)^{1/2}.
\end{equation}
However, we obtain a much better upper bound for the integral on the right hand side of \eqref{sk4} than the upper bound $c\cdot \delta^{3/2}$ in \eqref{sk3}  in the case of the equidistant  quasi-Milstein scheme $\widehat X^{\delta, eq}$. More precisely, we show that  there exists $c\in(0, \infty)$ such that for all $\delta\in(0, \delta_0]$,
\begin{equation}\label{sk5}
\int_0^1 \EE\bigl[ |\eul_t-\eul_{\utn}|^2\cdot \ind_S (\eul_t,\eul_{\utn})\bigr]\, dt\leq c\cdot\delta^{2},
\end{equation}
which jointly with \eqref{sk4} yields  the error estimate \eqref{ll33}. 
For the proof of \eqref{sk5} we split the integral on the left hand side  of \eqref{sk5}  into four terms   
using the identities 
\[
1=1_{(\Theta^{\varepsilon_0})^c}(\eul_{\utn})+1_{\Theta^{\varepsilon_0}\setminus \Theta^{\varepsilon_1^{\delta}}}(\eul_{\utn})+1_{\Theta^{\varepsilon_1^{\delta}}\setminus \Theta^{\varepsilon_2^{\delta}}}(\eul_{\utn})+ 1_{\Theta^{\varepsilon_2^{\delta}}}(\eul_{\utn}), \quad t\in [0, 1], 
\]
and prove the upper bound $c\cdot\delta^{2}$ for each of the resulting terms 
employing uniform $L_p$-estimates of  $\eul$, appropriate upper bounds  for the probabilities that the increments $|\eul_t-\eul_{\utn}|$  are  large 
compared to the distance of  $\eul_{\utn}$ from the set $\Theta$ as well as estimates for the expected value of certain occupation time functionals of $\eul$.
We add that for the proof of \eqref{sk5} it is crucial that the adaptive quasi-Milstein scheme $\eul$ uses smaller step sizes when it is close to the discontinuity points of $\mu$.

For the proof of the  estimate \eqref{ll32} we proceed similarly to the cost analysis of the  adaptive Euler-Maruyama scheme in~\cite[Section 5]{NS19}. 

We briefly describe the 
structure
of this section. 
In Section~\ref{4.0} we provide  properties of the random times $\tau_i^{\delta}$ and $\utn$ that are crucial for our proofs.
In Section~\ref{4.1} we prove $L_p$-estimates of the adaptive quasi-Milstein scheme
 $\eul$. Section~\ref{4.2} contains  estimates for the expected value of occupation time functionals of $\eul$ as well as  estimates for the probabilities that the increments $|\eul_t-\eul_{\utn}|$ of the adaptive quasi-Milstein scheme are  large 
compared to the distance of the actual value of the scheme $\eul_{\utn}$ from the set $\Theta$, which finally lead to the proof of the estimate \eqref{sk5},
see Proposition~\ref{prop1}.  The results in Sections~\ref{4.1} and~\ref{4.2} are then used in Section~\ref{4.3} to derive the error estimate  \eqref{ll33} in Theorem~\ref{Thm1}. Section \ref{cost} is devoted to the proof of the estimate  \eqref{ll32} in Theorem~\ref{Thm1}.

Throughout the following we will employ the following facts, which are  an immediate consequence of the assumptions  
($\mu$1), ($\mu$2)  and ($\sigma$1), ($\sigma$2) and the assumption that $\mu$ is continuous. Namely, the function $\mu$ is Lipschitz continuous on $\R$,
the functions $\mu$ and $\sigma$ satisfy a linear growth condition, i.e. 
\begin{equation}\label{LG}
\exists\, K\in (0, \infty)\,\forall\, x\in\R\colon\quad |\mu(x)|+|\sigma(x)|\leq K\cdot (1+|x|),
\end{equation}
the functions  $d_\mu$ and $d_\sigma$ are bounded, i.e. 
\begin{equation}\label{bound}
\|d_\mu\|_\infty + \|d_\sigma\|_\infty < \infty,
\end{equation}
and it holds
\begin{equation}\label{taylor}
\begin{aligned}
& \exists\, c\in(0, \infty)\,\forall\,f\in\{\mu, \sigma\} \,\forall\, 
i\in\{1,\dots,k+1\}\,
\forall x,y\in
 (\xi_{i-1},\xi_{i})
\colon \\
& \qquad \qquad\qquad |f(y)-f(x)-f'(x)(y-x)|  \le c\cdot |y-x|^2.
\end{aligned}
\end{equation}

\subsection{Properties of the random times $\tau_i^{\delta}$ and $\utn$.}\label{4.0}
Let $(\mathcal F_t)_{t\geq 0}$ denote the augmentation of the filtration generated by $W$, i.e. for all $t\geq 0$,
\[
\mathcal F_t= \sigma\bigl(\sigma(\{W_s\colon s\in [0,t]\})\cup \mathcal N\bigr),
\]
where $\mathcal N = \{N\in \mathcal F\colon \PP(N)=0\}$. 
For a stopping time $\tau\colon\Omega\to[0,\infty)$ let $\mathcal F_{\tau}$ denote the $\sigma$-algebra of $\tau$-past, i.e.
\[
\mathcal F_{\tau}=\{A\in\mathcal F\colon   A\cap \{\tau\leq t\}\in \mathcal F_t \text{ for all } t\geq 0\}.
\]
Moreover, for a random time $\tau\colon\Omega\to[0,\infty)$  define a stochastic process $W^{\tau}\colon [0, \infty)\times\Omega\to\R$ by
\[
W^{\tau}_t=W_{\tau+t}-W_{\tau}, \quad t\geq 0.
\]
The following two  lemmas provide the properties of the random times $\tau_i^{\delta}$ and $\utn$ that are crucial for our proofs.
\begin{lemma}\label{Stime} Let $\delta\in(0,\delta_0]$. Then for all $i\in\N_0$, 
\begin{itemize}
\item[(i)] $\tau_i^{\delta}$ is a stopping time and $\eul_{\tau_i^{\delta}}$ is $\mathcal F_{\tau_i^{\delta}}/ \mathcal B(\R)$-measurable,
\item[(ii)] $\tau_{i+1}^{\delta}$ is 
 $\mathcal F_{\tau_i^{\delta}}/ \mathcal B([0, \infty))$-measurable,
\item[(iii)] $W^{\tau_i^{\delta}}$ is a Brownian motion and independent of $\mathcal F_{\tau_i^{\delta}}$
\end{itemize}
and 
\begin{itemize}
\item[(iv)] $\tau_i^{\delta}\wedge 1$ is a stopping time and $\eul_{\tau_i^{\delta}\wedge 1}$ is $\mathcal F_{\tau_i^{\delta}\wedge 1}/ \mathcal B(\R)$-measurable,
\item[(v)] $\tau_{i+1}^{\delta}\wedge 1$ is 
 $\mathcal F_{\tau_i^{\delta}\wedge 1}/ \mathcal B([0, \infty))$-measurable,
\item[(vi)] $W^{\tau_i^{\delta}\wedge 1}$ is a Brownian motion and independent of $\mathcal F_{\tau_i^{\delta}\wedge 1}$.
\end{itemize}
\end{lemma}

\begin{proof}
We prove (i) by induction on $i\in\N_0$. Clearly, (i) holds for $i=0$. Next, assume that (i) holds  for some $i\in\N_0$.  Then using the definition \eqref{mil2} of $\tau_{i+1}^{\delta}$ we conclude that $\tau_{i+1}^{\delta}$ is $\mathcal F_{\tau_i^{\delta}}/ \mathcal B([0,\infty))$-measurable and $\tau_{i+1}^{\delta}\geq \tau_{i}^{\delta}$. Applying~\cite[Exercise 1.2.14]{ks91} we thus obtain that $\tau_{i+1}^{\delta}$ is a stopping time. This in particular yields that
 $W_{\tau_{i+1}^{\delta}}$ is $\mathcal F_{\tau_{i+1}^{\delta}}/ \mathcal B(\R)$-measurable and $W_{\tau_{i}^{\delta}}$ is $\mathcal F_{\tau_{i}^{\delta}}/ \mathcal B(\R)$-measurable. Thus, using the fact that $\mathcal F_{\tau_{i}^{\delta}}\subset \mathcal F_{\tau_{i+1}^{\delta}}$ as well as the induction assumption we  obtain from  the definition \eqref{mil2} of $\eul_{\tau_{i+1}^{\delta}}$ 
that $\eul_{\tau_{i+1}^{\delta}}$ is $\mathcal F_{\tau_{i+1}^{\delta}}/ \mathcal B(\R)$-measurable.
The  definition \eqref{mil2} of $\tau_{i+1}^{\delta}$ and (i) imply (ii). The strong Markov property of $W$ yields (iii). 

For the proof of (iv)-(vi) put
\[
s_i^\delta=\tau_i^{\delta}\wedge 1, \quad i\in\N_0,
\]
observe that $s_0^{\delta}=0, \eul_{s_0^{\delta}}=x_0$ 
and
\begin{align*}
s_{i+1}^{\delta}&= (s_i^{\delta}+h^{\delta}(\eul_{s_i^{\delta}}))\wedge 1,\\
\eul_{s_{i+1}^{\delta}}&=\eul_{s_i^{\delta}}+\mu(\eul_{s_i^{\delta}})\cdot (s_{i+1}^{\delta}-s_i^{\delta})+\sigma(\eul_{s_i^{\delta}})\cdot (W_{s_{i+1}^{\delta}}-W_{s_i^{\delta}})\\
&\qquad\qquad+\frac{1}{2}\sigma d_\sigma (\eul_{s_i^{\delta}})\cdot\bigl((W_{s_{i+1}^{\delta}}-W_{s_i^{\delta}})^2-(s_{i+1}^{\delta}-s_i^{\delta})\bigr)
\end{align*}
for $i\in\N_0$
and proceed similarly to the proof of (i)-(iii).

\end{proof}

\begin{lemma}\label{Stime1} Let $\delta\in(0,\delta_0]$ and  $t\in[0, \infty)$. Then $W^{\utn}$ is a Brownian motion and independent of $\eul_{\utn}$.
\end{lemma}
\begin{proof}
Clearly, $W^{\utn}$ is continuous. Employing Lemma \ref{Stime}(i),(ii),(iii) we obtain that for all $A\in \mathcal B(C([0, \infty);\R))$,
\begin{align*}
\PP(W^{\utn}\in A)&=\sum_{i=0}^{\infty}\PP(W^{\tau_i^{\delta}}\in A,\,\tau_i^{\delta}\leq t<\tau_{i+1}^{\delta})=\sum_{i=0}^{\infty}\PP(W^{\tau_i^{\delta}}\in A)\cdot \PP( \tau_i^{\delta}\leq t<\tau_{i+1}^{\delta})
=\PP(W\in A). 
\end{align*}
Thus, $W^{\utn}$ is a Brownian motion. Applying the latter fact as well as Lemma \ref{Stime}(i),(ii),(iii) we conclude that for all $A\in\mathcal B(C([0, \infty);\R))$ and all $B\in\mathcal B(\R)$,
\begin{align*}
\PP(W^{\utn}\in A,\,\eul_{\utn}\in B)&=\sum_{i=0}^{\infty}\PP(W^{\tau_i^{\delta}}\in A,\,\eul_{\tau_i^{\delta}}\in B,\, \tau_i^{\delta}\leq t<\tau_{i+1}^{\delta})\\
&=\sum_{i=0}^{\infty}\PP(W^{\tau_i^{\delta}}\in A)\cdot \PP( \eul_{\tau_i^{\delta}}\in B,\,\tau_i^{\delta}\leq t<\tau_{i+1}^{\delta})\\
&=\PP(W^{\utn}\in A)\cdot \sum_{i=0}^{\infty}\PP(\eul_{\tau_i^{\delta}}\in B,\, \tau_i^{\delta}\leq t<\tau_{i+1}^{\delta})\\
&=\PP(W^{\utn}\in A)\cdot \PP(\eul_{\utn}\in B),
\end{align*}
which shows that $W^{\utn}$  and  $\eul_{\utn}$ are independent and completes the proof of the lemma.
\end{proof}

\subsection{$L_p$ estimates  of the adaptive quasi-Milstein scheme}\label{4.1}

Using Lemma \ref{Stime}(i) one can show in a straightforward way that for all  $\delta\in(0,\delta_0]$ and all $t\in[0, \infty)$, 
\begin{equation}\label{intrep}
\eul_{t}=x_0+\int_0^t \mu(\eul_{\usn})\, ds+\int_0^t\bigl(\sigma(\eul_{\usn})+\sigma d_\sigma(\eul_{\usn})\cdot (W_s-W_{\usn})\bigr)\,dW_s \qquad \PP\text{-a.s.}
\end{equation}

Employing \eqref{intrep} we obtain the following uniform $L_p$-estimates for $\eul$, $\delta\in(0,\delta_0]$.

\begin{lemma}\label{eulprop}
Let $p\in[1, \infty)$. Then there exists  $c\in(0, \infty)$ such that for all  $\delta\in(0,\delta_0]$, 
\begin{equation}\label{mpr1}
 \EE\bigl[\sup_{t\in[0, 1]}  |\eul_t|^p\bigr]^{1/p}\leq c.
\end{equation}
Moreover, there exists  $c\in(0, \infty)$ such that for all  $\delta\in(0,\delta_0]$, all $\Delta\in[0,1]$ and all $t\in[0, 1-\Delta]$,
\begin{equation}\label{mpr2}
\EE\bigl[ \sup_{s\in[t, t+\Delta]}  |\eul_{s}-\eul_{t}|^p\bigr]^{1/p}\leq c\cdot \sqrt{\Delta}.
\end{equation}
\end{lemma}
\begin{proof}

We first show that  for all $\delta\in(0,\delta_0]$ and  all $i\in\N_0$,
\begin{equation}\label{w1}
\EE\bigl[|\eul_{\tau_i^{\delta}}|^p\bigr]<\infty.
\end{equation}
Let $\delta\in(0,\delta_0]$. We prove \eqref{w1} by induction on $i\in\N_0$. Clearly, \eqref{w1} holds for $i=0$. Next, assume that \eqref{w1} holds for some $i\in \N_0$. By \eqref{mil2}, \eqref{LG} and \eqref{bound} there exist $c_1, c_2\in(0, \infty)$ such that
\begin{align*}
|\eul_{\tau_{i+1}^{\delta}}|^p&\leq c_1\cdot \bigl(|\eul_{\tau_i^{\delta}}|^p+|\mu(\eul_{\tau_i^{\delta}})|^p\cdot \delta^p+|\sigma(\eul_{\tau_i^{\delta}})|^p\cdot |W_{\tau_{i+1}^{\delta}}-W_{\tau_i^{\delta}}|^p\\
&\qquad\qquad+\frac{1}{2}|\sigma d_\sigma (\eul_{\tau_i^{\delta}})|^p\cdot(|W_{\tau_{i+1}^{\delta}}-W_{\tau_i^{\delta}}|^{2p}+\delta^p)\bigr)\\
&\leq c_2\cdot (1+|\eul_{\tau_i^{\delta}}|^p)\cdot (1+\sup_{t\in[0, \delta]}|W_t^{\tau_{i}^{\delta}}|^{2p}+\sup_{t\in[0, \delta]}|W_t^{\tau_{i}^{\delta}}|^{p}).
\end{align*}
Using the independence of $\eul_{\tau_i^{\delta}}$ and $W^{\tau_{i}^{\delta}}$, the fact that $W^{\tau_{i}^{\delta}}$ is a Brownian motion as well as the induction assumption we therefore conclude that 
\begin{align*}
\EE\bigl[|\eul_{\tau_{i+1}^{\delta}}|^p\bigr]
&\leq c_2\cdot (1+\EE\bigl[|\eul_{\tau_i^{\delta}}|^p\bigr])\cdot (1+\EE\bigl[\sup_{t\in[0, \delta]}|W_t^{\tau_{i}^{\delta}}|^{2p}\bigr]+\EE\bigl[\sup_{t\in[0, \delta]}|W_t^{\tau_{i}^{\delta}}|^{p}\bigr])<\infty,
\end{align*}
 which completes the proof of \eqref{w1}.

 For $\delta\in(0,\delta_0]$ put
\begin{equation}\label{ndelta}
n^{\delta}=\lceil\delta^{-2}\log^{-4}(1/\delta)\rceil.
\end{equation}
It follows from \eqref{w1} that for all $\delta\in(0,\delta_0]$,
\begin{equation}\label{n1}
\sup_{t\in[0,1]} \EE\bigl[|\eul_{\utn}|^p\bigr]=\sup_{t\in[0,1]} \sum_{i=0}^{n^{\delta}}\EE\bigl[|\eul_{\tau_i^{\delta}}|^p\cdot 1_{\{\utn=\tau_i^{\delta}\}}\bigr]\leq\sum_{i=0}^{n^{\delta}}\EE\bigl[|\eul_{\tau_i^{\delta}}|^p\bigr]<\infty.
\end{equation}

We next prove \eqref{mpr1}. By \eqref{intrep}, for all $\delta\in(0,\delta_0]$ and  all $t\in[0,1]$,
\begin{align*}
\EE\bigl[\sup_{s\in[0, t]} |\eul_{s}|^p\bigr]\leq 3^p\cdot |x_0|^p&+3^p\cdot\EE\Bigl[\Bigl|\int_0^t |\mu(\eul_{\uun})|\, du\Bigr|^p\Bigr]\\
&+3^p\cdot\EE\Bigl[\sup_{s\in[0, t]}\Bigl|\int_0^s\bigl(\sigma(\eul_{\uun})+\sigma\cdot d_{\sigma} (\eul_{\uun})\cdot (W_u-W_{\uun})\bigr)\,dW_u\Bigr|^p\Bigr].
\end{align*}
Using the H\"older inequality, the Burkholder-Davis-Gundy inequality,
\eqref{LG} and  \eqref{bound}  we conclude that there exists $c\in(0, \infty)$ such that  for all $\delta\in(0,\delta_0]$ and  all $t\in[0,1]$,
\begin{equation}\label{w6}
\EE\bigl[\sup_{s\in[0, t]} |\eul_{s}|^p\bigr]\leq c+c\cdot \int_0^t\EE\bigl[|\eul_{\uun}|^p\bigr]\, du+c\cdot \int_0^t\EE\bigl[(1+|\eul_{\uun}|^p)\cdot |W_u-W_{\uun}|^p\bigr]\, du.
\end{equation}
Lemma \ref{Stime1} implies that there exists $c\in(0, \infty)$ such that for all $\delta\in(0,\delta_0]$ and  all $u\in[0,1]$,
\begin{equation}\label{w5}
\begin{aligned}
\EE\bigl[(1+|\eul_{\uun}|^p)\cdot |W_u-W_{\uun}|^p\bigr]&\leq \EE\bigl[(1+|\eul_{\uun}|^p)\cdot \sup_{s\in[0, \delta]}|W^{\uun}_s|^p\bigr]\\
&= \EE\bigl[(1+|\eul_{\uun}|^p)\bigr]\cdot \EE\bigl[\sup_{s\in[0, \delta]}|W_s|^p\bigr]\leq c\cdot \EE\bigl[1+|\eul_{\uun}|^p\bigr].
\end{aligned}
\end{equation}
Combining \eqref{w6} and \eqref{w5} we conclude that there exists $c\in(0, \infty)$ such that  for all $\delta\in(0,\delta_0]$ and  all $t\in[0,1]$,
\begin{equation}\label{n3}
\EE\bigl[\sup_{s\in[0, t]} |\eul_{s}|^p\bigr]\leq c+c\cdot \int_0^t\EE\bigl[|\eul_{\uun}|^p\bigr]\, du.
\end{equation}
Employing \eqref{n1} we therefore obtain that for all $\delta\in(0,\delta_0]$,
\begin{equation}\label{n2}
\EE\bigl[\sup_{s\in[0, 1]} |\eul_{s}|^p\bigr]<\infty.
\end{equation}
Moreover, by \eqref{n3},  for all $\delta\in(0,\delta_0]$ and  all $t\in[0,1]$,
\[
\EE\bigl[\sup_{s\in[0, t]} |\eul_{s}|^p\bigr]\leq c+c\cdot \int_0^t\EE\bigl[\sup_{u\in[0, s]}|\eul_{u}|^p\bigr]\, ds.
\]
Applying the Gronwall inequality completes the proof of \eqref{mpr1}.

For the proof of \eqref{mpr2} observe that for all $\delta\in(0,\delta_0]$, all $\Delta\in[0,1]$ and all $t\in[0, 1-\Delta]$,
\begin{align*}
\EE\bigl[ \sup_{s\in[t, t+\Delta]}  |\eul_{s}-\eul_{t}|^p\bigr]&\leq  2^p\cdot\EE\Bigl[\Bigl|\int_t^{t+\Delta} |\mu(\eul_{\uun})|\, du\Bigr|^p\Bigr]\\&\quad+2^p\cdot\EE\Bigl[\sup_{s\in[t, t+\Delta]}\Bigl|\int_t^s\bigl(\sigma(\eul_{\uun})+\sigma\cdot d_{\sigma} (\eul_{\uun})\cdot (W_u-W_{\uun})\bigr)\,dW_u\Bigr|^p\Bigr]
\end{align*}
and employ the H\"older inequality, the Burkholder-Davis-Gundy inequality,
\eqref{LG},  \eqref{bound}, \eqref{w5} and (i).
\end{proof}

\subsection{Occupation time  estimates for the adaptive quasi-Milstein scheme}\label{4.2} We first provide an estimate for the expected value of occupation time functionals of $\eul$. 
\begin{lemma}\label{occup}
Let $f\colon [0, \infty)\to [0, \infty)$ be  $\mathcal B([0, \infty))/\mathcal B([0, \infty))$-measurable and let $\gamma>0$. Then  there exists $c\in (0, \infty)$ such that for all $\delta\in(0, \delta_0]$ and all $\varepsilon\in(0, \varepsilon_0]$,
\[
\EE\Bigl[\int_0^1 f(d(\eul_t, \Theta))\cdot 1_{\Theta^\varepsilon}(\eul_t)\, dt\Bigr]\leq c\cdot\int_0^\varepsilon f(x)\,dx+c \cdot \sup_{x\in[0, \varepsilon]} f(x) \cdot(\varepsilon^{\frac{3}{2}-\gamma}+\delta^{\frac{3}{2}-\gamma}\bigr).
\]
\end{lemma}

\begin{proof} Clearly, it is enought to show that for all $i\in\{1, \ldots, k\}$ there exists $c\in (0, \infty)$ such that for all $\delta\in(0, \delta_0]$ and all $\varepsilon\in(0, \varepsilon_0]$,
\begin{equation}\label{oo1}
\EE\Bigl[\int_0^1 f(|\eul_t-\xi_i|)\cdot 1_{[\xi_i-\varepsilon, \xi_i+\varepsilon]}(\eul_t)\, dt\Bigr]\leq c\cdot\int_0^\varepsilon f(x)\,dx+c \cdot \sup_{x\in[0, \varepsilon]} f(x)\cdot (\varepsilon^{\frac{3}{2}-\gamma}+\delta^{\frac{3}{2}-\gamma}\bigr).
\end{equation}
In the following fix $i\in\{1, \ldots, k\}$.

Let   $\delta\in(0, \delta_0]$. For $t\in[0,1]$ put
\[
\Sigma^\delta_{t}=\sigma(\eul_{\utn})+\sigma d_\sigma (\eul_{\utn})\cdot (W_t-W_{\utn}).
\]
Using~\eqref{LG},~\eqref{bound},
\eqref{intrep}
and Lemma~\ref{eulprop} we conclude that $\eul$ is a continuous semi-martingale
 with quadratic variation 
\begin{equation}\label{qv}
\langle \eul\rangle_t
=\int_0^t (\Sigma^\delta_{s})^2\, ds,\quad t\in[0,1].
\end{equation}
For $a\in\R$ let $L^a(\eul) = (L^a_t(\eul))_{t\in[0,1]}$ denote the local time
of $\eul$ at the point $a$.
Thus, 
for all $a\in\R$ and
 all $t\in[0,1]$, 
\begin{align*}
|\eul_{t}-a| & = |x_0-a| + \int_0^t \sgn(\eul_{s}-a)\cdot \mu (\eul_{\usn})\, ds + \int_0^t \sgn(\eul_{s}-a)\cdot \Sigma^\delta_{s} \, dW_s + L^a_t(\eul),
\end{align*}
where $\sgn(y) = 1_{(0,\infty)}(y) - 1_{(-\infty,0]}(y)$ for $y\in\R$,
see, e.g.~\cite[Chap.~VI]{RevuzYor2005}.
Hence, 
for all $a\in\R$ and
 all $t\in[0,1]$,
 \begin{equation}\label{jjj0}
\begin{aligned}
L^a_t(\eul) & \le |\eul_{t}-x_0| + \int_0^t |\mu (\eul_{\usn})|\, ds + \Bigl|\int_0^t \sgn(\eul_{s}-a)\cdot \Sigma^\delta_{s} \, dW_s\Bigr|\\
&\leq 2\int_0^t |\mu (\eul_{\usn})|\, ds +\Bigl|\int_0^t  \Sigma^\delta_{s} \, dW_s\Bigr|+ \Bigl|\int_0^t \sgn(\eul_{s}-a)\cdot \Sigma^\delta_{s} \, dW_s\Bigr|.
\end{aligned}
\end{equation}

Using~\eqref{LG},~\eqref{jjj0}, the H\"older inequality, the Burkholder-Davis-Gundy inequality and Lemma~\ref{eulprop} we obtain that
there exist $c_1, c_2\in (0,\infty)$ such that
  for all  $\delta\in(0, \delta_0]$,
all $a\in\R$ 
   and all $t\in[0,1]$,
\begin{equation}\label{jjj1}
\begin{aligned}
\EE\bigl[L^a_t(\eul)\bigr]  &\le c_1\cdot \int_0^1 \bigl(1+\EE\bigl[|\eul_{\usn}|\bigr]\bigr)\, ds+c_1 \,\Bigl(\int_0^1  \EE\bigl[(\Sigma^\delta_{s})^2\bigr] \, ds\Bigr)^{1/2}\\
&\leq c_2+c_1 \,\Bigl(\int_0^1  \EE\bigl[(\Sigma^\delta_{s})^2\bigr] \, ds\Bigr)^{1/2}.
\end{aligned}
\end{equation}  
Moreover, by \eqref{LG},~\eqref{bound},
 Lemma \ref{Stime1} and Lemma~\ref{eulprop}   there exist $c_1, c_2\in(0,\infty)$ such that for all $s\in[0,1]$ and all  $\delta\in(0, \delta_0]$,
\begin{equation}\label{jjj2}
\begin{aligned}
\EE\bigl[(\Sigma^\delta_{s})^2\bigr]&\leq c_1\cdot \EE\bigl[(1+|\eul_{\usn}|)^2\cdot (1+|W_s-W_{\usn}|)^2\bigr]\\
&\leq c_1\cdot\EE\bigl[(1+|\eul_{\usn}|)^2]\cdot \EE[(1+\sup_{u\in[0, \delta]}|W^{\usn}_u|)^2\bigr]\leq c_2.
\end{aligned}
\end{equation}
Combining~\eqref{jjj1} and~\eqref{jjj2} we obtain 
that there exists 
$c\in (0, \infty)$ such that
  for all  $\delta\in(0, \delta_0]$,
all $a\in\R$ 
   and all $t\in[0,1]$,
\begin{equation}\label{local1}
\EE\bigl[L^a_t(\eul)\bigr] 
\leq  c.
\end{equation}
Using~\eqref{qv},~\eqref{local1} and the occupation 
time
formula it follows that
there exists $c\in (0,\infty)$ such that  
for all $\delta\in(0, \delta_0]$ and all $\eps\in (0,\varepsilon_0]$,
\begin{equation}\label{local2}
\begin{aligned}
 & \EE\biggl[\int_0^1 f(|\eul_t-\xi_i|)\cdot 1_{[\xi_i-\eps,\xi_i+\eps]}(\eul_{t})\cdot (\Sigma^\delta_{t})^2\, dt\biggr]\\
 &\qquad\qquad= \int_{\R} f(|a-\xi_i|)\cdot 1_{[\xi_i-\eps,\xi_i+\eps]}(a)\cdot \EE\bigl[L^a_t(\eul)\bigr]\, da \le c\cdot \int_{0}^\eps f(x)\,dx. 
 \end{aligned}
\end{equation}

By ~\eqref{LG},~\eqref{bound}
and the Lipschitz continuity of $\sigma$
we obtain that 
there exist  $c_1,c_2\in (0,\infty)$ such that 
for all $\delta\in(0, \delta_0]$ and all $t\in[0,1]$,
\begin{align*}
\bigl|\sigma^2(\eul_{t})-(\Sigma^\delta_{t})^2\bigr|&\leq \bigl|\sigma(\eul_{t})-\Sigma^\delta_{t}\bigr|\cdot \bigl(|\sigma(\eul_{t})|+|\Sigma^\delta_{t}|\bigr)\\
&\leq c_1\cdot \bigl(|\sigma(\eul_{t})-\sigma(\eul_{\utn})|+|\sigma \delta_\sigma (\eul_{\utn})|\cdot |W_t-W_{\utn}|\bigr)\\
&\qquad\,\cdot \bigl(1+|\eul_{t}|+(1+|\eul_{\utn}|)\cdot(1+|W_t-W_{\utn}|)\bigr)\\
&\leq c_2\cdot  \bigl(|\eul_{t}-\eul_{\utn}|+(1+|\eul_{\utn}|)\cdot \sup_{u\in[0, \delta]}|W^{\utn}_u|\bigr) \\
&\qquad\,\cdot(1+\sup_{s\in[0,1]}|\eul_{s}|)\cdot(1+\sup_{u\in[0, \delta]}|W^{\utn}_u|).
\end{align*}
Thus, using the H\"older inequality,
 Lemma~\ref{eulprop} and Lemma \ref{Stime1} we conclude that for all $q\in[1, \infty)$ there exists  $c\in (0,\infty)$ such that 
for all $\delta\in(0, \delta_0]$ and all $t\in[0,1]$,
\begin{equation}\label{local3} 
\EE\bigl[|\sigma^2(\eul_{t})-(\Sigma^\delta_{t})^2|^q\bigr]^{1/q}\leq c\cdot \sqrt\delta.
\end{equation}

Since $\sigma$ is continuous and $\sigma(\xi_i)\neq 0$ there exist  $\kappa_i,\rho_i\in(0,\infty)$ such that 
\begin{equation}\label{localx} 
\inf_{x\in\R: |x-\xi_i|\leq \rho_i}\sigma^2(x) \ge \kappa_i.
\end{equation}
Using~\eqref{local2},~\eqref{local3}, ~\eqref{localx} and the H\"older inequality we obtain that for all $q\in(1, \infty)$ there exists $c\in (0,\infty)$ such that for all $\delta\in(0, \delta_0]$ and all $\eps \in (0,\rho_i\wedge \eps_0]$,
\begin{equation}\label{hhhh}
\begin{aligned}
&\EE\Bigl[\int_0^1 f(|\eul_t-\xi_i|)\cdot 1_{[\xi_i-\varepsilon, \xi_i+\varepsilon]}(\eul_t)\, dt\Bigr] \\
&  \qquad \le \frac{1}{\kappa}\cdot \EE\Bigl[\int_0^1 f(|\eul_t-\xi_i|)\cdot 1_{[\xi_i-\varepsilon, \xi_i+\varepsilon]}(\eul_t)\cdot  \sigma^2(\eul_{t})\, dt\Bigr]\\
&  \qquad \le \frac{1}{\kappa}\cdot \EE\Bigl[\int_0^1 f(|\eul_t-\xi_i|)\cdot 1_{[\xi_i-\varepsilon, \xi_i+\varepsilon]}(\eul_t)\cdot  (\Sigma^\delta_{t})^2\, dt\Bigr]\\
 &\qquad \qquad+\frac{1}{\kappa}\cdot \EE\Bigl[\int_0^1 f(|\eul_t-\xi_i|)\cdot 1_{[\xi_i-\varepsilon, \xi_i+\varepsilon]}(\eul_t)\cdot  \bigl|\sigma^2(\eul_{t})-(\Sigma^\delta_{t})^2\bigr|\, dt\Bigr]\\
&   \qquad\le  c\cdot \int_{0}^\eps f(x)\,dx+c \cdot \sup_{x\in[0, \varepsilon]} f(x)\cdot \sqrt \delta \cdot \int_0^1 \bigl(\PP(|\eul_t-\xi_i|\leq \varepsilon)\bigr)^{1/q}\, dt\\
 &\qquad\le  c\cdot \int_{0}^\eps f(x)\,dx+c \cdot \sup_{x\in[0, \varepsilon]} f(x)\cdot \sqrt \delta\cdot \Bigl( \int_0^1 \PP(|\eul_t-\xi_i|\leq \varepsilon)\, dt\Bigl)^{1/q}.
\end{aligned}  
\end{equation}
Note that in the case of $f=1$ the estimate \eqref{hhhh} yields that for all $q\in(1, \infty)$ there exists $c\in (0,\infty)$ such that for all $\delta\in(0, \delta_0]$ and all $\eps \in (0,\rho_i\wedge \eps_0]$,
\[
 \int_0^1 \PP(|\eul_t-\xi_i|\leq \varepsilon)\, dt\leq c\cdot \varepsilon+c\cdot \sqrt\delta\cdot  \Bigl( \int_0^1 \PP(|\eul_t-\xi_i|\leq \varepsilon)\, dt\Bigl)^{1/q}.
\]
Thus, observing that $\varepsilon_0\in (0,1]$ and $\delta_0\in (0,1)$ and using the Young inequality we obtain that for all $q\in(1, 2]$ there exist $c_1, c_2\in (0,\infty)$ such that for all $\delta\in(0, \delta_0]$ and all $\eps \in (0,\rho_i\wedge \eps_0]$,
\begin{equation}\label{pp1}
 \int_0^1 \PP(|\eul_t-\xi_i|\leq \varepsilon)\, dt\leq c_1\cdot \varepsilon+c_1\cdot \sqrt\delta\cdot  \Bigl( c_1\cdot \varepsilon+c_1\cdot \sqrt\delta\Bigl)^{1/q}\leq c_2\cdot \varepsilon+c_2\cdot \delta^{\frac{1}{2}+\frac{1}{2q}}.
\end{equation}
It follows from \eqref{hhhh} and \eqref{pp1}  that for all $q\in(1, 2]$ there exists $c\in (0,\infty)$ such that for all $\delta\in(0, \delta_0]$ and all $\eps \in (0,\rho_i\wedge \eps_0]$,
\begin{align*}
&\EE\Bigl[\int_0^1 f(|\eul_t-\xi_i|)\cdot 1_{[\xi_i-\varepsilon, \xi_i+\varepsilon]}(\eul_t)\, dt\Bigr] \le  c\cdot \int_{0}^\eps f(x)\,dx+c \cdot \sup_{x\in[0, \varepsilon]} f(x)\cdot \sqrt \delta \cdot \bigl(\varepsilon^{\frac{1}{q}}+\delta^{\frac{1}{2q}+\frac{1}{2q^2}}\bigr).
\end{align*}
By the Young inequality, for all $q\in(1, 2]$, all $\delta\in(0, \delta_0]$ and all $\eps \in (0,\rho_i\wedge \eps_0]$,
\[
\sqrt \delta \cdot \varepsilon^{\frac{1}{q}}\leq \frac{1}{3}\delta^{3/2}+\frac{2}{3}\varepsilon^{\frac{3}{2q}}.
\]
Combining the latter two estimates we conclude that for all $q\in(1, 2]$ there exists $c\in (0,\infty)$ such that for all $\delta\in(0, \delta_0]$ and all $\eps \in (0,\rho_i\wedge \eps_0]$,
\begin{align*}
&\EE\Bigl[\int_0^1 f(|\eul_t-\xi_i|)\cdot 1_{[\xi_i-\varepsilon, \xi_i+\varepsilon]}(\eul_t)\, dt\Bigr] \le  c\cdot \int_{0}^\eps f(x)\,dx+c \cdot \sup_{x\in[0, \varepsilon]} f(x)\cdot  \bigl(\varepsilon^{\frac{3}{2q}}+\delta^{\frac{1}{2}+\frac{1}{2q}+\frac{1}{2q^2}}\bigr).
\end{align*}
This yields \eqref{oo1} and  completes the proof of the lemma.

\end{proof}

The following lemma provides upper bounds for the probabilities that increments of the adaptive quasi-Milstein scheme are  large 
compared to the actual distance of the scheme from the set $\Theta$. 

\begin{lemma}\label{dist}
Let $\alpha\in(0, \infty)$ and $q\in[1, \infty)$. Then there exist $c_1, c_2, c_3\in (0, \infty)$ such that for all $\delta\in(0,\delta_0]$ and all $t\in[0,1]$,
\begin{itemize}
\item[(i)]
$\PP(|\eul_{\utn}-\eul_{t}|\geq \alpha\cdot \varepsilon_2^{\delta},\, \eul_{\utn}\in \Theta^{\varepsilon_2^{\delta}})\leq c_1\cdot \delta^q$,
\item[(ii)]
$\PP(|\eul_{t}-\eul_{\utn}|\geq \alpha\cdot  d(\eul_{\utn}, \Theta),\, \eul_{\utn}\in \Theta^{\varepsilon_1^{\delta}}\setminus \Theta^{\varepsilon_2^{\delta}})\leq c_2\cdot \delta^q$,
\item[(iii)]
$\PP(|\eul_{t}-\eul_{\utn}|\geq \alpha\cdot \varepsilon_1^{\delta},\,\eul_{\utn}\in \Theta^{\varepsilon_0}\setminus \Theta^{\varepsilon_1^{\delta}})\leq c_3\cdot \delta^q$.
\end{itemize}

\end{lemma}
\begin{proof}
Define $\Phi\colon\R\times C([0, \infty);\R)\to C([0, \infty);\R)$ by
\[
\Phi(y, u)(t)=y+\mu(y)\cdot t+\sigma(y)\cdot u(t)+
\frac{1}{2}\sigma d_\sigma (y)\cdot(u^2(t)-t)
\]
for $y\in\R$, $u\in C([0, \infty);\R)$ and $t\in[0, \infty)$ and  observe that there exists $\kappa\in (0, \infty)$ such that for all $y\in \Theta^{\varepsilon_0}$, all $u\in C([0, \infty);\R)$ and all $t\in[0, \infty)$,
\begin{equation}\label{Phi}
|\Phi(y, u)(t)-y|\leq \kappa\cdot (t+|u(t)|+
u^2(t)).
\end{equation}

We first proof (i). 
Using Lemma \ref{Stime1} we obtain that for all $\delta\in(0, \delta_0]$ and all $t\in[0,1]$, 
\begin{equation}\label{e33b} 
\begin{aligned}
&\PP(|\eul_{\utn}-\eul_{t}|\geq \alpha\cdot \varepsilon_2^{\delta},\, \eul_{\utn}\in \Theta^{\varepsilon_2^{\delta}})\\
&\qquad= \PP(|\Phi(\eul_{\utn},W^{\utn})(t-\utn)-\eul_{\utn}|\geq \alpha\cdot \varepsilon_2^{\delta}, \,\eul_{\utn}\in \Theta^{\varepsilon_2^{\delta}})\\
&\qquad\leq \PP(\sup_{s\in[0, h^{\delta}(\eul_{\utn})]}|\Phi(\eul_{\utn},W^{\utn})(s)-\eul_{\utn}|\geq \alpha\cdot \varepsilon_2^{\delta}, \,\eul_{\utn}\in \Theta^{\varepsilon_2^{\delta}})\\
&\qquad=\int_{\Theta^{\varepsilon_2^{\delta}}} \PP(\sup_{s\in[0, h^{\delta}(y)]}|\Phi(y,W)(s)-y|\geq \alpha\cdot \varepsilon_2^{\delta} )\,\PP^{\eul_{\utn}}(dy).
\end{aligned}
\end{equation}
By \eqref{Phi},  for all $\delta\in(0, \delta_0]$ and all $y\in \Theta^{\varepsilon_2^{\delta}}$,
\begin{equation}\label{e44b}
\begin{aligned}
&\PP(\sup_{s\in[0, h^{\delta}(y)]}|\Phi(y,W)(s)-y|\geq \alpha\cdot \varepsilon_2^{\delta}  )\\
&\qquad\leq \PP\bigr(h^{\delta}(y)+\sup_{s\in[0, h^{\delta}(y)]}|W_s|+\sup_{s\in[0, h^{\delta}(y)]}W_s^2\geq \tfrac{\alpha \varepsilon_2^{\delta}}{\kappa} \bigl)\\
&\qquad \leq \PP\bigr(h^{\delta}(y)+\sup_{s\in[0, h^{\delta}(y)]}|W_s|\geq \tfrac{\alpha\varepsilon_2^{\delta}}{2\kappa} \bigl)+\PP\bigr(\sup_{s\in[0, h^{\delta}(y)]}W_s^2\geq \tfrac{\alpha\varepsilon_2^{\delta}}{2\kappa} \bigl)\\
&\qquad=\PP\bigr(\sup_{s\in[0, h^{\delta}(y)]}|W_s|\geq \tfrac{\alpha\varepsilon_2^{\delta}}{2\kappa} -h^{\delta}(y)\bigl)+\PP\bigr(\sup_{s\in[0, h^{\delta}(y)]}|W_s|\geq \tfrac{\sqrt{\alpha\varepsilon_2^{\delta}}}{\sqrt{2\kappa}}  \bigl).
\end{aligned}
\end{equation}
Recall that for all $\delta\in(0, \delta_0]$ and all  $y\in  \Theta^{\varepsilon_2^{\delta}}$ we have
$
h^{\delta}(y)=\delta^2 \log^4(1/\delta)
$.
Moreover, by~\cite[Lemma 3.4]{NSS19}, 
there exists $c\in (0, \infty)$ such that for all $u\in(0,\infty)$ and all $x\in\R$,
\begin{equation}\label{tail}
\PP(\sup_{s\in[0, u]}|W_s|\geq x)\leq c\cdot e^{-\frac{x}{\sqrt u}}.
\end{equation}
Hence, there exist $c_1, c_2, c_3 \in(0, \infty)$ such that for all $\delta\in(0, \delta_0]$ and all  $y\in  \Theta^{\varepsilon_2^{\delta}}$,
\begin{align*}
\PP\bigl(\sup_{s\in[0, h^{\delta}(y)]}|W_s|\geq \tfrac{\alpha\varepsilon_2^{\delta}}{2\kappa} -h^{\delta}(y)\bigr)&\leq c_1\cdot e^{-\tfrac{\alpha}{2\kappa}\log^2(1/\delta)+\delta\log^2(1/\delta)}\leq c_2\cdot \delta^{q}
\end{align*}
as well as
\begin{align*}
\PP\bigl(\sup_{s\in[0, h^{\delta}(y)]}|W_s|\geq \tfrac{\sqrt{\alpha\varepsilon_2^{\delta}}}{\sqrt{2\kappa}}\bigr)&\leq c_1\cdot e^{-\tfrac{\sqrt{\alpha}}{\sqrt{2\kappa\delta}}}\leq c_3\cdot \delta^{q}.
\end{align*}
The latter two estimates together with \eqref{e33b} and \eqref{e44b} imply (i).

We next proof (ii). Proceeding similarly to \eqref{e33b} and \eqref{e44b}  we obtain that  for all $\delta\in(0, \delta_0]$ and all $t\in[0,1]$,
\begin{equation}\label{e33} 
\begin{aligned}
&
\PP(|\eul_{t}-\eul_{\utn}|\geq \alpha\cdot d(\eul_{\utn}, \Theta),\,\eul_{\utn}\in \Theta^{\varepsilon_1^{\delta}}\setminus \Theta^{\varepsilon_2^{\delta}})\\
&\qquad\leq \int_{\Theta^{\varepsilon_1^{\delta}}\setminus \Theta^{\varepsilon_2^{\delta}}} \PP(\sup_{s\in[0, h^{\delta}(y)]}|\Phi(y,W)(s)-y|\geq \alpha\cdot d(y, \Theta) )\,\PP^{\eul_{\utn}}(dy)
\end{aligned}
\end{equation}
and  for all $\delta\in(0, \delta_0]$ and all  $y\in \Theta^{\varepsilon_1^{\delta}}\setminus \Theta^{\varepsilon_2^{\delta}}$,
\begin{equation}\label{e44}
\begin{aligned}
&\PP(\sup_{s\in[0, h^{\delta}(y)]}|\Phi(y,W)(s)-y|\geq \alpha \cdot d(y, \Theta) )\\
&\qquad\leq \PP\bigr(\sup_{s\in[0, h^{\delta}(y)]}|W_s|\geq \tfrac{\alpha}{2\kappa}\, d(y, \Theta) -h^{\delta}(y)\bigl)+\PP\bigr(\sup_{s\in[0, h^{\delta}(y)]}|W_s|\geq \tfrac{\sqrt{\alpha}}{\sqrt{2\kappa}}\, \sqrt{d(y, \Theta)} \bigl).
\end{aligned}
\end{equation}
Recall that for all $\delta\in(0, \delta_0]$ and all  $y\in \Theta^{\varepsilon_1^{\delta}}\setminus \Theta^{\varepsilon_2^{\delta}}$ we have
$
h^{\delta}(y)=\Bigl(\frac{d(y, \Theta)}{\log^2(1/\delta)}\Bigr)^2.
$
Hence, applying \eqref{tail} we obtain that there exist $c_1, c_2, c_3 \in(0, \infty)$ such that for all $\delta\in(0, \delta_0]$ and all  $y\in \Theta^{\varepsilon_1^{\delta}}\setminus \Theta^{\varepsilon_2^{\delta}}$,
\begin{align*}
\PP\bigr(\sup_{s\in[0, h^{\delta}(y)]}|W_s|\geq \tfrac{\alpha}{2\kappa}\, d(y, \Theta) -h^{\delta}(y)\bigl)\leq c_1\cdot e^{-\tfrac{\alpha}{2\kappa }\log^2(1/\delta)  +\sqrt {h^{\delta}(y)}}\leq c_2\cdot \delta^{q}
\end{align*}
as well as
\begin{align*}
\PP\bigr(\sup_{s\in[0, h^{\delta}(y)]}|W_s|\geq \tfrac{\sqrt{\alpha}}{\sqrt{2\kappa}}\, \sqrt{d(y, \Theta)} \bigl)\leq c_1\cdot e^{-\tfrac{\sqrt{\alpha} }{\sqrt{2\kappa} }\cdot \tfrac{\log^2(1/\delta)}{\sqrt{ d(y, \Theta)}}}\leq c_1\cdot e^{-\tfrac{\sqrt{\alpha} }{\sqrt{2\kappa} }\cdot \tfrac{\log(1/\delta)}{\delta^{1/4}}}\leq c_3\cdot \delta^{q}.
\end{align*}
The latter two estimates together with \eqref{e33} and \eqref{e44} yield (ii).

We finally prove (iii). Proceeding similarly to \eqref{e33b} and \eqref{e44b}  we obtain that  for all $\delta\in(0, \delta_0]$ and all $t\in[0,1]$,
\begin{equation}\label{e33a} 
\begin{aligned}
&\PP(|\eul_{t}-\eul_{\utn}|\geq \alpha\cdot \varepsilon_1^{\delta},\,\eul_{\utn}\in \Theta^{\varepsilon_0}\setminus \Theta^{\varepsilon_1^{\delta}})\\
&\qquad\leq \int_{\Theta^{\varepsilon_0}\setminus \Theta^{\varepsilon_1^{\delta}}} \PP(\sup_{s\in[0, h^\delta(y)]}|\Phi(y,W)(s)-y|\geq \alpha\cdot \varepsilon_1^{\delta} )\,\PP^{\eul_{\utn}}(dy)
\end{aligned}
\end{equation}
and  for all $\delta\in(0, \delta_0]$ and all  $y\in \Theta^{\varepsilon_0}\setminus \Theta^{\varepsilon_1^{\delta}}$,
\begin{equation}\label{e44a}
\begin{aligned}
&\PP(\sup_{s\in[0, h^\delta(y)]}|\Phi(y,W)(s)-y|\geq \alpha\cdot \varepsilon_1^{\delta} )\\
&\qquad\leq \PP\bigr(\sup_{s\in[0, h^\delta(y)]}|W_s|\geq \tfrac{\alpha \varepsilon_1^{\delta}}{2\kappa} -h^\delta(y)\bigl)+\PP\bigr(\sup_{s\in[0, h^\delta(y)]}|W_s|\geq \tfrac{\sqrt{\alpha \varepsilon_1^{\delta}}}{\sqrt{2\kappa}}  \bigl).
\end{aligned}
\end{equation}
Recall that for all $\delta\in(0, \delta_0]$ and all  $y\in \Theta^{\varepsilon_0}\setminus \Theta^{\varepsilon_1^{\delta}}$ we have
$
h^{\delta}(y)=\delta
$.
Applying  \eqref{tail} we therefore obtain that there exist $c_1, c_2, c_3 \in(0, \infty)$ such that for all $\delta\in(0, \delta_0]$ and all $y\in \Theta^{\varepsilon_0}\setminus \Theta^{\varepsilon_1^{\delta}}$,
\begin{align*}
\PP\bigl(\sup_{s\in[0, h^\delta(y)]}|W_s|\geq \tfrac{\alpha\varepsilon_1^{\delta}}{2\kappa}-h^{\delta}(y)\bigr)&\leq c_1\cdot e^{-\tfrac{\alpha}{2\kappa}\log^2(1/\delta)+\sqrt\delta}
\leq c_2\cdot \delta^{q}
\end{align*}
as well as
\begin{align*}
\PP\bigl(\sup_{s\in[0, h^{\delta}(y)]}|W_s|\geq \tfrac{\sqrt{\alpha \varepsilon_1^{\delta}}}{\sqrt{2\kappa}}\bigr)\leq c_1\cdot e^{-\tfrac{ \sqrt{\alpha}}{\sqrt{2\kappa}}\cdot\tfrac{\log(1/\delta)}{\delta^{1/4}}}\leq c_3\cdot \delta^{q}.
\end{align*}
The latter two estimates together with \eqref{e33a} and \eqref{e44a} imply (iii) and complete the proof of the lemma.

\end{proof}

Next, put
\begin{equation}\label{setS}
S = \Bigl(\bigcup_{\ell=1}^{k+1} (\xi_{\ell-1},\xi_\ell)^2\Bigr)^c
\end{equation}
and  note that $S=\cup_{\ell=1}^{k}\{(x,y)\in \R^2\colon (x-\xi_\ell)\cdot (y-\xi_\ell)\le 0\}$. 
We are ready to establisch the main result in this section, which provides a $p$-th mean estimate of the time average of $|\eul_{t}-\eul_{\utn}|^2$ subject to the condition that the pair $(\eul_{t},\eul_{\utn})$ lies in the set $S$.
\begin{prop}\label{prop1}
Let  $p\in [1,\infty)$. Then  there exists  $c\in(0, \infty)$ such that for all $\delta\in(0,\delta_0]$, 
\begin{equation}\label{l33a}
\EE\Bigl[\Bigl|\int_0^1  |\eul_{t}-\eul_{\utn}|^2\cdot \ind_{S} (\eul_{t},\eul_{\utn})\,dt\Bigr|^p\Bigr]^{1/p}\leq c\cdot \delta^{2}. 
\end{equation}
\end{prop}
\begin{proof}
For $\delta\in(0,\delta_0]$ and $i\in\{1,2,3,4\}$ let
\[
E^{\delta}_{i}=\EE\Bigl[\int_0^1  |\eul_{t}-\eul_{\utn}|^{2p}\cdot \ind_{S} (\eul_{t},\eul_{\utn})\cdot \ind_{O^{\delta}_{i}}(\eul_{\utn}) \,dt\Bigr],
\]
where 
\[
O_1^{\delta}=(\Theta^{\varepsilon_0})^c, \quad O_2^{\delta}=\Theta^{\varepsilon_0}\setminus \Theta^{\varepsilon_1^{\delta}}, \quad O_3^{\delta}=\Theta^{\varepsilon_1^{\delta}}\setminus \Theta^{\varepsilon_2^{\delta}} , \quad O_4^{\delta}=\Theta^{\varepsilon_2^{\delta}}.
\]
Then for all $\delta\in(0,\delta_0]$, 
\begin{equation}\label{q1}
\EE\Bigl[\Bigl|\int_0^1  |\eul_{t}-\eul_{\utn}|^2\cdot \ind_{S} (\eul_{t},\eul_{\utn})\,dt\Bigr|^p\Bigr]\leq \EE\Bigl[\int_0^1  |\eul_{t}-\eul_{\utn}|^{2p}\cdot \ind_{S} (\eul_{t},\eul_{\utn})\,dt\Bigr]\leq \sum_{i=1}^4 E^{\delta}_{i}.
\end{equation}
Below we show that for all $i\in\{1,2,3,4\}$ there exists $c\in(0, \infty)$ such that for all $\delta\in(0,\delta_0]$, 
\begin{equation}\label{q2}
E^{\delta}_{i} \leq c\cdot \delta^{2p}.
\end{equation}
Clearly, \eqref{q1} and \eqref{q2} imply \eqref{l33a}.

It remains to prove \eqref{q2}. We start with the analysis of $E^{\delta}_{1}$. For all $\delta\in(0,\delta_0]$ and all $t\in[0,1]$,
\[
\{(\eul_{t},\eul_{\utn})\in S\}\cap \{\eul_{\utn}\in O^{\delta}_{1}\}\subseteq \{|\eul_{t}-\eul_{\utn}|\geq \varepsilon_0\}.
\]
Thus, using the Markov inequality and Lemma \ref{eulprop} we obtain that there exist $c_1, c_2\in(0, \infty)$ such that for all $\delta\in(0,\delta_0]$,
\begin{align*}
E^{\delta}_{1}& \leq \int_0^1  \EE\bigl[|\eul_{t}-\eul_{\utn}|^{2p}\cdot \ind_{\{|\eul_{t}-\eul_{\utn}|\geq \varepsilon_0\}}\bigr]\,dt\\
&\leq \int_0^1  \EE\bigl[|\eul_{t}-\eul_{\utn}|^{4p}\bigr]^{1/2}\cdot (\PP(|\eul_{t}-\eul_{\utn}|\geq \varepsilon_0))^{1/2} \,dt\leq \frac{1}{\varepsilon_0^{2p}}\int_0^1  \EE\bigl[|\eul_{t}-\eul_{\utn}|^{4p}\bigr] \,dt\\
&\leq \frac{c_1}{\varepsilon_0^{2p}}\int_0^1  \bigl(\EE\bigl[|\eul_{t}-\eul_{0\vee (t-\delta)}|^{4p}\bigr]+\EE\bigl[|\eul_{\utn}-\eul_{0\vee( t-\delta)}|^{4p}\bigr]\bigr) \,dt\\
&\leq \frac{2c_1}{\varepsilon_0^{2p}}\int_0^1  \EE\bigl[\sup_{s\in[0\vee (t-\delta), t]}|\eul_{s}-\eul_{0\vee (t-\delta)}|^{4p}\bigr] \,dt\leq c_2\cdot\delta^{2p},
\end{align*}
which shows that \eqref{q2} holds for $i=1$.

We next estimate $E^{\delta}_{2}$.  Using Lemma \ref{eulprop} we obtain that there exists $c\in(0, \infty)$ such that  for all $\delta\in(0,\delta_0]$,
\begin{equation}\label{e4}
\begin{aligned}
E^{\delta}_{2}
& \leq \int_0^1  \EE\bigl[|\eul_{t}-\eul_{\utn}|^{4p}\bigr]^{1/2}\cdot (\PP((\eul_{t},\eul_{\utn})\in S, \,\eul_{\utn}\in O_2^\delta))^{1/2} \,dt\\
&\leq c\cdot \delta^{p}\cdot \int_0^1  (\PP((\eul_{t},\eul_{\utn})\in S,\,\eul_{\utn}\in \Theta^{\varepsilon_0}\setminus \Theta^{\varepsilon_1^{\delta}}))^{1/2} \,dt.
\end{aligned}
\end{equation}
Moreover, using Lemma \ref{dist}(iii) with $\alpha=1$ and $q=2p$ we conclude that there exists $c\in(0, \infty)$ such that for all $\delta\in(0, \delta_0]$ and all $t\in[0,1]$,
\begin{align*}
\PP((\eul_{t},\eul_{\utn})\in S, \,\eul_{\utn}\in \Theta^{\varepsilon_0}\setminus \Theta^{\varepsilon_1^{\delta}})\leq \PP(|\eul_{t}-\eul_{\utn}|\geq \varepsilon_1^{\delta},\,\eul_{\utn}\in \Theta^{\varepsilon_0}\setminus \Theta^{\varepsilon_1^{\delta}})\leq c\cdot \delta^{2p}.
\end{align*}
The latter estimate together with \eqref{e4} yields \eqref{q2} for $i=2$.

We next estimate $E^{\delta}_{3}$. Similarly to \eqref{e4} we obtain 
that there exists $c\in(0, \infty)$ such that  for all $\delta\in(0,\delta_0]$,
\begin{equation}\label{e6}
\begin{aligned}
E^{\delta}_{3}\leq c\cdot \delta^{p}\cdot \int_0^1  (\PP((\eul_{t},\eul_{\utn})\in S,\,\eul_{\utn}\in \Theta^{\varepsilon_1^{\delta}}\setminus \Theta^{\varepsilon_2^{\delta}}))^{1/2} \,dt.
\end{aligned}
\end{equation}
Moreover, using Lemma \ref{dist}(ii) with $\alpha=1$ and $q=2p$ we conclude that there exists $c\in(0, \infty)$ such that for all $\delta\in(0, \delta_0]$ and all $t\in[0,1]$,
\[
\PP((\eul_{t},\eul_{\utn})\in S, \,\eul_{\utn}\in \Theta^{\varepsilon_1^{\delta}}\setminus \Theta^{\varepsilon_2^{\delta}})\leq \PP(|\eul_{t}-\eul_{\utn}|\geq d(\eul_{\utn}, \Theta),\,\eul_{\utn}\in \Theta^{\varepsilon_1^{\delta}}\setminus \Theta^{\varepsilon_2^{\delta}}) \leq c\cdot \delta^{2p}.
\]
The latter estimate together with \eqref{e6} yields \eqref{q2} for $i=3$.

We finally extimate $E^{\delta}_{4}$. Note that for all $\delta\in(0,\delta_0]$, all $t\in[0,1]$ and all
$\omega\in \{\eul_{\utn}\in \Theta^{\varepsilon_2^{\delta}}\}$ we have $t-\utn(\omega)\leq \delta^2\cdot\log^4(1/\delta)$. 
Thus, using Lemma \ref{eulprop} we obtain that there exist $c_1, c_2\in(0, \infty)$ such that for all $\delta\in(0,\delta_0]$,
\begin{equation}\label{mm1}
\begin{aligned}
E^{\delta}_{4} &\leq \int_0^1  \EE\bigl[|\eul_{t}-\eul_{\utn}|^{2p}\cdot \ind_{\Theta^{\varepsilon_2^{\delta}}}(\eul_{\utn})\bigr]\,dt\\
&\leq \int_0^1  \EE\bigl[|\eul_{t}-\eul_{\utn}|^{4p}\cdot \ind_{\Theta^{\varepsilon_2^{\delta}}}(\eul_{\utn})\bigr]^{1/2}\cdot (\PP(\eul_{\utn}\in \Theta^{\varepsilon_2^{\delta}}))^{1/2}\,dt\\
&\leq c_1\int_0^1  \EE\bigl[\sup_{s\in[0\vee(t-\delta^2\cdot\log^4(1/\delta)), t]}|\eul_{s}-\eul_{0\vee (t-\delta^2\cdot\log^4(1/\delta))}|^{4p}\bigr]^{1/2}\cdot (\PP(\eul_{\utn}\in \Theta^{\varepsilon_2^{\delta}}))^{1/2}\,dt\\
&\leq c_2\cdot\delta^{2p}\cdot \log^{4p}(1/\delta)\cdot \int_0^1 (\PP(\eul_{\utn}\in \Theta^{\varepsilon_2^{\delta}}))^{1/2}\,dt\\
&\leq c_2\cdot\delta^{2p}\cdot \log^{4p}(1/\delta)\cdot \Bigl(\int_0^1 \PP(\eul_{\utn}\in \Theta^{\varepsilon_2^{\delta}})\,dt\Bigr)^{1/2}.
\end{aligned}
\end{equation} 
Employing Lemma \ref{occup} with $f=1$ and $\gamma=1/2$ and Lemma \ref{dist}(i) with $\alpha=1$ and $q=1$ we obtain that there exist $c\in(0, \infty)$ such that for all $\delta\in(0,\delta_0]$,
\begin{equation}\label{ww}
\begin{aligned}
\int_0^1\PP(\eul_{\utn}\in \Theta^{\varepsilon_2^{\delta}})\,dt&=\int_0^1\PP(|\eul_{t}-\eul_{\utn}|<\varepsilon_2^{\delta},\, \eul_{\utn}\in \Theta^{\varepsilon_2^{\delta}})\, dt\\
&\qquad\qquad+\int_0^1\PP(|\eul_{t}-\eul_{\utn}|\geq \varepsilon_2^{\delta}, \,\eul_{\utn}\in \Theta^{\varepsilon_2^{\delta}})\, dt\\
& \leq \int_0^1\PP(\eul_{t}\in \Theta^{2\varepsilon_2^{\delta}})\, dt+\int_0^1\PP(|\eul_{t}-\eul_{\utn}|\geq \varepsilon_2^{\delta}, \, \eul_{\utn}\in \Theta^{\varepsilon_2^{\delta}})\,dt\\
&\leq c\cdot \varepsilon_2^{\delta}+ c \cdot \delta\leq 2c\cdot \delta\cdot\log^4(1/\delta).
\end{aligned}
\end{equation}
The latter estimate together with \eqref{mm1} implies that there exist $c_1,c_2\in(0, \infty)$ such that for all $\delta\in(0,\delta_0]$,
\[
E^{\delta}_{4}\leq c_1\cdot \delta^{2p+1/2}\cdot \log^{4p+2}(1/\delta)\leq c_2\cdot\delta^{2p},
\]
which shows that \eqref{q2} holds for $i=4$ and completes the proof of the proposition.
\end{proof}

\subsection{Convergence analysis}\label{4.3} In this subsection we proof the estimate \eqref{ll33}.
Clearly, it is enough to consider the case $p\in\N\setminus\{1\}$.
 For $\delta\in(0,\delta_0]$ and $t\in [0,1]$ we put 
\[
A_t = \int_0^t \mu(X_s)\, ds,\quad \widehat A_{t}^{\delta}  = \int_0^t  \mu(\eul_{\usn})\, ds 
\]
and
\[
B_t = \int_0^t \sigma(X_s)\, dW_s,\quad \widehat B_{t} = \int_0^t \bigl(\sigma(\eul_{\usn}) + \sigma d_\sigma(\eul_{\usn})\cdot (W_s-W_{\usn})\bigr)\, dW_s
\]
as well as
\[
U_{t}^{\delta} = \int_0^t \sigma d_\mu(\eul_{\usn})\cdot (W_s-W_{\usn})\, ds
\]
and we use the decomposition
\begin{equation}\label{end1}
X_t -\eul_{t} = (A_t-\widehat A_{t}^{\delta}-U_{t}^{\delta}) +  (B_t-\widehat B_{t}^{\delta}) + U_{t}^{\delta}.
\end{equation}

Recall the definition \eqref{setS} of the set $S$. For all $\delta\in(0, \delta_0]$, all  $s\in [0,1]$ and all $f\in\{\mu, \sigma\}$ we have
\begin{align*}
& |f(X_s) - f(\eul_{\usn}) - \sigma d_f(\eul_{\usn})\cdot (W_s-W_{\usn})| \\  & \qquad \le |f(X_s) - f(\eul_{s})|\ + |f(\eul_{s})-f(\eul_{\usn})-d_f(\eul_{\usn})\cdot(\eul_{s}-\eul_{\usn})|\cdot \ind_{S^c}(\eul_{s},\eul_{\usn})\\
& \qquad \qquad + |f(\eul_{s})-f(\eul_{\usn})-d_f(\eul_{\usn})\cdot(\eul_{s}-\eul_{\usn})|\cdot \ind_{S}(\eul_{s},\eul_{\usn})\\
 & \qquad \qquad + \bigl|d_f(\eul_{\usn})\cdot ( \mu(\eul_{\usn})(s-\usn) + \tfrac{1}{2}\sigma d_\sigma(\eul_{\usn})\cdot ((W_s-W_{\usn})^2 -(s-\usn)))\bigr|.
\end{align*}
Using the Lipschitz continuity of $\mu$ and $\sigma$ as well as~\eqref{LG},~\eqref{bound} and~\eqref{taylor} we thus obtain that there exists $c\in (0,\infty)$ such that for all $\delta\in(0,\delta_0]$, all  $s\in [0,1]$ and all $f\in\{\mu, \sigma\}$, 
\begin{equation}\label{end4}
\begin{aligned}
& |f(X_s) - f(\eul_{\usn}) - \sigma d_f(\eul_{\usn})\cdot (W_s-W_{\usn})| \\ 
& \qquad \le c\cdot |X_s - \eul_{s}| +  c\cdot |\eul_{s} - \eul_{\usn}|^2 + c\cdot |\eul_{s} - \eul_{\usn}|\cdot \ind_{S}(\eul_{s},\eul_{\usn})\\
& \qquad \qquad + c\cdot (1+|\eul_{\usn}|)\cdot (\delta + |W_s-W_{\usn}|^2).
\end{aligned}
\end{equation}

 Employing~\eqref{end4}, Lemma~\ref{eulprop} and Proposition \ref{prop1}  we conclude that there exist $c_1,c_2, c_3\in (0,\infty)$ such that for all $\delta\in(0,\delta_0]$ and all $t\in [0,1]$,
\begin{equation}\label{end5}
\begin{aligned}
& \EE\Bigl[\,\sup_{0\le s\le t}|A_s-\widehat A_{s}^{\delta}-U_{s}^{\delta}|^p\Bigr] \\
& \qquad \qquad\le \EE\Bigl[\Bigl|\int_0^t |\mu(X_s) - \mu(\eul_{\usn}) - \sigma d_\mu(\eul_{\usn})\cdot (W_s-W_{\usn})|\, ds\Bigr|^p\Bigr]\\
& \qquad\qquad \le c_1\cdot \int_0^t \EE\bigl[|X_s - \eul_{s}|^p\bigr]\, ds + c_1\cdot \int_0^t \EE\bigl[|\eul_{s} - \eul_{\usn}|^{2p}\bigr]\, ds\\
& \qquad \qquad \qquad+ c_1\cdot \EE\Bigl[\Bigl|\int_0^t |\eul_{s}-\eul_{\usn}|\cdot \ind_{S} (\eul_{s},\eul_{\usn})\, ds\Bigr|^p\Bigr] \\
& \qquad \qquad \qquad+ c_1\cdot \int_0^t \EE\bigl[1+\sup_{u\in[0,1]}|\eul_{u}|^{2p}\bigr]^{1/2}\cdot \EE\bigl[\delta^{2p} + \sup_{u\in[0\vee (s-\delta), s]}|W_s-W_{u}|^{4p}\bigr]^{1/2}\, ds\\
& \qquad\qquad \le c_1\cdot \int_0^t \EE\bigl[|X_s - \eul_{s}|^p\bigr]\, ds +c_1\cdot \EE\Bigl[\Bigl|\int_0^t |\eul_{s}-\eul_{\usn}|^2\cdot \ind_{S} (\eul_{s},\eul_{\usn})\, ds\Bigr|^{p/2}\Bigr]\\
&\qquad \qquad \qquad+c_2\cdot \int_0^t \EE\bigl[\sup_{u\in[0\vee (s-\delta), s]}|\eul_{u} - \eul_{0\vee (s-\delta)}|^{2p}\bigr]\, ds+ c_2\cdot\delta^p\\
& \qquad\qquad \le c_1\cdot \int_0^t \EE\bigl[|X_s - \eul_{s}|^p\bigr]\, ds + c_3\cdot\delta^p.
\end{aligned}
\end{equation}

Using the Burkholder-Davis-Gundy inequality, \eqref{end4},  Lemma~\ref{eulprop} and Proposition \ref{prop1}  we obtain that there exist $c_1,c_2,c_3\in (0,\infty)$ such that for all $\delta\in(0,\delta_0]$ and all $t\in [0,1]$,
\begin{equation}\label{end7}
\begin{aligned}
 &\EE\Bigl[\,\sup_{0\le s\le t}|B_s-\widehat B_{s}^{\delta}|^p\Bigr]\\
  &  \qquad\qquad \le c_1\cdot \EE\Bigl[\Bigl|\int_0^t |\sigma(X_s) - \sigma(\eul_{\usn}) - \sigma d_\sigma(\eul_{\usn})\cdot (W_s-W_{\usn})|^2\, ds\Bigr|^{p/2}\Bigr] \\
& \qquad\qquad  \le c_2\cdot \int_0^t \EE\bigl[|X_s - \eul_{s}|^p\bigr]\, ds + c_2\cdot \int_0^t \EE\bigl[|\eul_{s} - \eul_{\usn}|^{2p}\bigr]\, ds\\
&  \qquad\qquad \qquad+ c_2\cdot \EE\Bigl[\Bigl|\int_0^t |\eul_{s}-\eul_{\usn}|^2\cdot \ind_{S} (\eul_{s},\eul_{\usn})\, ds\Bigr|^{p/2}\Bigr] \\
&  \qquad\qquad  \qquad+ c_2\cdot \int_0^t \EE\bigl[1+\sup_{u\in[0,1]}|\eul_{u}|^{2p}\bigr]^{1/2}\cdot \EE\bigl[\delta^{2p} + \sup_{u\in[0\vee( s-\delta), s]}|W_s-W_{u}|^{4p}\bigr]^{1/2}\, ds\\
&   \qquad\qquad \le c_2\cdot \int_0^t \EE\bigl[|X_s - \eul_{s}|^p\bigr]\, ds + c_3\cdot\delta^p.
\end{aligned}
\end{equation}

Combining~\eqref{end1} with~\eqref{end5} and~\eqref{end7} we conclude that there exists $c\in (0,\infty)$ such that for 
all $\delta\in(0,\delta_0]$ and
all $t\in[0,1]$,
\begin{equation}\label{end8}
\begin{aligned}
\EE\bigl[\,\sup_{0\le s\le t}|X_t-\eul_{t}|^p\bigr] \le c\cdot \int_0^t \EE\bigl[\sup_{0\le u\le s}|X_u - \eul_{u}|^p\bigr]\, ds + c\cdot\delta^p + \EE\bigl[\,\sup_{0\le s\le t}|U_{s}^{\delta}|^p\bigr].
\end{aligned}
\end{equation}
Note that $\EE\bigl[\sup_{0\le u\le 1}|X_u-\eul_u|^p\bigr] < \infty$ due to~\eqref{mom} and Lemma~\ref{eulprop}.
Below we show that there exists $c\in (0,\infty)$ such that for all $\delta\in(0,\delta_0]$,
\begin{equation}\label{end9}
\begin{aligned}
\EE\bigl[\,\sup_{0\le s\le 1}|U_{s}^{\delta}|^p\bigr] \le c\cdot \delta^p.
\end{aligned}
\end{equation}
Inserting~\eqref{end9} into~\eqref{end8} and applying the Gronwall inequality then yields the error estimate \eqref{ll33} in Theorem~\ref{Thm1}.

We turn to the proof of~\eqref{end9}. 
Clearly, for all $\delta\in(0, \delta_0]$, all $i\in\N_0$ and all $s\in [\tau_{i}^{\delta}, \tau_{i+1}^{\delta}]$,
\begin{equation}\label{end10}
\begin{aligned}
U^{\delta}_{s\wedge 1} &= U^{\delta}_{\tau_{i}^{\delta}\wedge 1} + \sigma d_\mu(\eul_{\tau_{i}^{\delta}\wedge 1})\cdot \int_{\tau_{i}^{\delta}\wedge 1}^{s\wedge 1}(W_u-W_{\tau_{i}^{\delta}})\, du.
\end{aligned}
\end{equation}
For $\delta\in(0,\delta_0]$ let $n^\delta$ be given by \eqref{ndelta}.
Using~\eqref{LG} and~\eqref{bound}  we obtain from~\eqref{end10} that there exists $c\in (0,\infty)$ such that for all $\delta\in(0,\delta_0]$,
\begin{equation}\label{end11}
\begin{aligned}
\sup_{0\le s\le 1} |U_{s}^{\delta}| &=\max_{i=0,\dots,n^{\delta}-1}\sup_{\tau_{i}^{\delta}\leq s \leq\tau_{i+1}^{\delta}} |U_{s\wedge 1}^{\delta}|\\
& \le \max_{i=0,\dots,n^{\delta}-1}|U_{\tau_{i}^{\delta}\wedge 1}^{\delta}| +\max_{i=0,\dots,n^{\delta}-1}|\sigma d_\mu(\eul_{\tau_{i}^{\delta}\wedge 1})|\cdot \int_{\tau_{i}^{\delta}\wedge 1}^{\tau_{i+1}^{\delta}\wedge 1}|W_u-W_{\tau_{i}^{\delta}}|\, du\\
& \le \max_{i=0,\dots,n^{\delta}-1}|U_{\tau_{i}^{\delta}\wedge 1}^{\delta}|+ c\cdot (1+\sup_{0\leq s\leq 1}|\eul_s|)\cdot \max_{i=0,\dots,n^{\delta}-1}\int_{\tau_{i}^{\delta}\wedge 1}^{\tau_{i+1}^{\delta}\wedge 1}|W_u-W_{\tau_{i}^{\delta}}|\, du.
\end{aligned}
\end{equation}

Let $\delta\in(0,\delta_0]$. Employing \eqref{end10}, \eqref{LG}, \eqref{bound} and Lemma~\ref{eulprop} one can show  by induction on  $i\in\{0,\dots,n^{\delta}-1\}$ that $\EE\bigl[|U^{\delta}_{\tau_{i}^{\delta}\wedge 1}\bigr|\bigr]<\infty$ for all $i\in\{0,\dots,n^{\delta}-1\}$. Moreover, using Lemma \ref{Stime}(vi),(v) one can show by induction on  $i\in\{0,\dots,n^{\delta}-1\}$ that $U^{\delta}_{\tau_{i}^{\delta}\wedge 1}$ is $\mathcal F_{\tau_{i}^{\delta}\wedge 1}/ \mathcal B(\R)$-measurable for all $i\in\{0,\dots,n^{\delta}-1\}$.
Finally, observe that for all $i\in\{0,\dots,n^{\delta}-2\}$,
\[
\int_{\tau_{i}^{\delta}\wedge 1}^{\tau_{i+1}^{\delta}\wedge 1}(W_u-W_{\tau_{i}^{\delta}})\, du
= \int_{0}^{(\tau_{i+1}^{\delta}\wedge 1)-(\tau_{i}^{\delta}\wedge 1)}W^{\tau_{i}^{\delta}\wedge 1}_u\, du.
\]
Using Lemma \ref{Stime}(vi),(v),(vi) we therefore obtain that for all $i\in\{0,\dots,n^{\delta}-2\}$,
\begin{align*}
\EE\bigl[U^{\delta}_{\tau_{i+1}^{\delta}\wedge 1}| \mathcal F_{\tau_{i}^{\delta}\wedge 1}\bigr]
&=U^{\delta}_{\tau_{i}^{\delta}\wedge 1} + \sigma d_\mu(\eul_{\tau_{i}^{\delta}\wedge 1})\cdot \int_{0}^{(\tau_{i+1}^{\delta}\wedge 1)-(\tau_{i}^{\delta}\wedge 1)} \EE\bigl[W^{\tau_{i}^{\delta}\wedge 1}_u| \mathcal F_{\tau_{i}^{\delta}\wedge 1}\bigr]\, du=U^{\delta}_{\tau_{i}^{\delta}\wedge 1}.
\end{align*}
Hence,  the sequence $(U^{\delta}_{\tau_{i}^{\delta}\wedge 1}, \mathcal F_{\tau_{i}^{\delta}\wedge 1})_{i\in\{0,\dots,n^{\delta}-1\}}$ is a martingale. 

Employing the Burkholder-Davis-Gundy inequality as well as~\eqref{LG},~\eqref{bound}, Lemma~\ref{eulprop} and Lemma \ref{Stime1} we conclude that there exist $c_1,c_2,c_3\in (0,\infty)$ such that for all $\delta\in(0,\delta_0]$,
\begin{equation}\label{end13}
\begin{aligned}
&\EE\bigl[\, \max_{i=0,\dots,n^{\delta}-1}|U_{\tau_{i}^{\delta}\wedge 1}^{\delta}|^p\Bigr] \\
& \qquad\le 
\EE\Bigl[\Bigl( \sum_{i=0}^{n^{\delta}-1} \Bigl(\sigma d_\mu(\eul_{\tau_{i}^{\delta}\wedge 1})\cdot \int_{\tau_{i}^{\delta}\wedge 1}^{\tau_{i+1}^{\delta}\wedge 1}(W_u-W_{\tau_{i}^{\delta}})\, du\Bigr)^2\Bigr)^{p/2}\Bigr]\\
&\qquad \le c_1\cdot \EE\bigl[(1+\sup_{0\leq s\leq 1}|\eul_s|^{2p})\bigr]^{1/2}\cdot \EE\Bigl[\Bigl( \sum_{i=0}^{n^{\delta}-1} \Bigl( \int_{\tau_{i}^{\delta}\wedge 1}^{\tau_{i+1}^{\delta}\wedge 1}(W_u-W_{\tau_{i}^{\delta}})\, du\Bigr)^2\Bigr)^{p}\Bigr]^{1/2}\\
&\qquad \le c_2 \cdot \EE\Bigl[\Bigl( \sum_{i=0}^{n^{\delta}-1} ((\tau_{i+1}^{\delta}\wedge 1)-(\tau_{i}^{\delta}\wedge 1))\cdot  \int_{\tau_{i}^{\delta}\wedge 1}^{\tau_{i+1}^{\delta}\wedge 1}(W_u-W_{\tau_{i}^{\delta}})^2\, du\Bigr)^{p}\Bigr]^{1/2}\\
&\qquad \leq c_2  \cdot \delta^{p/2}\cdot \EE\Bigl[\Bigl(  \int_{0}^{1}(W_u-W_{\uun})^2\, du\Bigr)^{p}\Bigr]^{1/2}\\
&\qquad \leq c_2  \cdot \delta^{p/2}\cdot \Bigl(  \int_{0}^{1}\EE \bigl[\sup_{s\in[0, \delta]}|W^{\uun}_s|^{2p}\bigr]\, du \Bigr)^{1/2}\leq c_3\cdot\delta^p.
\end{aligned}
\end{equation}
Furthermore, using Lemma~\ref{eulprop} and Lemma \ref{Stime1} we obtain that there exists
$c_1\in (0,\infty)$
such that for all $\delta\in(0,\delta_0]$,
\begin{equation}\label{end14}
\begin{aligned}
& \EE\Bigl[\Bigl((1+\sup_{0\leq s\leq 1}|\eul_s|)
\cdot \max_{i=0,\dots,n^{\delta}-1}\int_{\tau_{i}^{\delta}\wedge 1}^{\tau_{i+1}^{\delta}\wedge 1}|W_u-W_{\tau_{i}^{\delta}}|\, du\Bigr)^p\Bigr] \\
& \qquad\qquad \le \EE\bigl[(1+\sup_{0\leq s\leq 1}|\eul_s|^{2p})\bigr]^{1/2}
\cdot\EE\Bigl[ \sum_{i=0}^{n^\delta-1} \Bigl( \int_{\tau_{i}^{\delta}\wedge 1}^{\tau_{i+1}^{\delta}\wedge 1}|W_u-W_{\tau_{i}^{\delta}}|\, du\Bigr)^{2p}\Bigr]^{1/2}\\
& \qquad\qquad \le c_1\cdot \EE\Bigl[ \sum_{i=0}^{n^\delta-1}  ((\tau_{i+1}^{\delta}\wedge 1)-(\tau_{i}^{\delta}\wedge 1))^{2p-1}\cdot  \int_{\tau_{i}^{\delta}\wedge 1}^{\tau_{i+1}^{\delta}\wedge 1}|W_u-W_{\tau_{i}^{\delta}}|^{2p}\, du\Bigr]^{1/2}\\
& \qquad\qquad \le c_1\cdot \delta^{\frac{2p-1}{2}}\cdot \Bigl(  \int_{0}^{1}\EE \bigl[\sup_{s\in[0, \delta]}|W^{\uun}_s|^{2p}\bigr]\, du \Bigr)^{1/2}\leq c_1\cdot \delta^{\frac{3p-1}{2}}\leq  c_1\cdot \delta^{p}.
\end{aligned}
\end{equation}
Combining~\eqref{end11} with~\eqref{end13} and~\eqref{end14} 
yields~\eqref{end9}
and completes the proof of the estimate~\eqref{ll33} in Theorem~\ref{Thm1}.

\subsection{Cost analysis.} \label{cost} In this subsection we proof the estimate \eqref{ll32}.
Clearly, for all  $\delta\in(0, \delta_0]$  and all $i\in\N$ we have
\[
1=\int_{\tau_{i-1}^{\delta}}^{\tau_{i}^{\delta}} \frac{1}{\tau_{i}^{\delta}-\tau_{i-1}^{\delta}}\,dt =\int_{\tau_{i-1}^{\delta}}^{\tau_{i}^{\delta}} \frac{1}{h^{\delta}(\eul_{\utn})}\,dt.
\]
Thus, for all $\delta\in(0, \delta_0]$,
\[
N(\eul_1)=1+\sum_{i=1}^\infty 1_{\{\tau_i^{\delta}<1\}}=1+\sum_{i=1}^\infty 1_{\{\tau_i^{\delta}<1\}}\cdot \int_{\tau_{i-1}^{\delta}}^{\tau_{i}^{\delta}} \frac{1}{h^{\delta}(\eul_{\utn})}\,dt \leq 1+\int_0^1\frac{1}{h^{\delta}(\eul_{\utn})}\,dt.
\]
For $\delta\in(0, \delta_0]$ and $i\in\{1,2,3\}$ put
\[
I_i^{\delta}=\EE\Bigl[\int_0^1\frac{1}{h^{\delta}(\eul_{\utn})}\cdot 1_{O_i^{\delta}}(\eul_{\utn})\,dt\Bigr],
\]
where
\[
O_1^{\delta}=(\Theta^{\varepsilon_1^{\delta}})^c, \quad O_2^{\delta}=\Theta^{\varepsilon_1^{\delta}}\setminus \Theta^{\varepsilon_2^{\delta}} , \quad O_3^{\delta}=\Theta^{\varepsilon_2^{\delta}}.
\]
Then for all $\delta\in(0, \delta_0]$,
\begin{equation}\label{ww2}
\EE[N(\eul_1)]\leq 1+\sum_{i=1}^3 I_i^{\delta}.
\end{equation}
Clearly,
\begin{equation}\label{ww3}
I_1^{\delta}=\delta^{-1}\cdot\int_0^1\PP\bigl(\eul_{\utn}\in (\Theta^{\varepsilon_1^{\delta}})^c\bigr)\,dt\leq \delta^{-1}.
\end{equation}
Moreover, observing \eqref{ww} we obtain that there exists $c\in(0, \infty)$ such that for all $\delta\in(0, \delta_0]$,
\begin{equation}\label{ww4}
I_3^{\delta}=\delta^{-2}\cdot\log^{-4}(1/\delta)\cdot\int_0^1\PP\bigl(\eul_{\utn}\in \Theta^{\varepsilon_2^{\delta}}\bigr)\,dt\leq c\cdot \delta^{-1}.
\end{equation}
Below we show that there exists $c\in(0, \infty)$ such that for all $\delta\in(0, \delta_0]$,
\begin{equation}\label{ww1}
I_2^{\delta}\leq c\cdot \delta^{-1}.
\end{equation}
Combining \eqref{ww2} to \eqref{ww1} we obtain \eqref{ll32}.

It remains to prove \eqref{ww1}. For  $\delta\in(0, \delta_0]$ and $t\in[0,1]$ put
\[
D_t^{\delta}=\bigl{\{}|\eul_{t}-\eul_{\utn}|\leq \tfrac{1}{2}d(\eul_{\utn}, \Theta)\bigr{\}}.
\]
Clearly, for all $\delta\in(0, \delta_0]$,
\begin{equation}\label{vv1}
I_{2}^{\delta}=I_{2,1}^{\delta}+I_{2,2}^{\delta},
\end{equation}
where
\[
I_{2,1}^{\delta}=\EE\Bigl[\int_0^1\frac{1}{h^{\delta}(\eul_{\utn})}\cdot 1_{O_2^{\delta}}(\eul_{\utn})\cdot 1_{D_t^{\delta}}\,dt\Bigr],\quad  I_{2,2}^{\delta}=\EE\Bigl[\int_0^1\frac{1}{h^{\delta}(\eul_{\utn})}\cdot 1_{O_2^{\delta}}(\eul_{\utn})\cdot 1_{(D_t^{\delta})^c}\,dt\Bigr].
\]
Observing the fact  that the distance function $d(\cdot, \Theta)\colon\R\to [0, \infty)$ is Lipschitz continuous with Lipschitz seminorm $1$, i.e. for all $x,y\in\R$,
\[
|d(x, \Theta)-d(y, \Theta)|\leq |x-y|,
\]
we obtain that for all $\delta\in(0, \delta_0]$ and all $t\in[0,1]$,
\[
\{\eul_{\utn}\in O_2^{\delta}\}\cap D_t^{\delta}\subseteq \{\eul_{t}\in \Theta^{\frac{3}{2}\varepsilon_1^{\delta}}\setminus \Theta^{\frac{1}{2}\varepsilon_2^{\delta}}\}\cap \bigl{\{}\tfrac{1}{2}d(\eul_{\utn}, \Theta)\leq d(\eul_t, \Theta)\leq \tfrac{3}{2}d(\eul_{\utn}, \Theta)\bigr{\}}.
\]
Thus, for all $\delta\in(0, \delta_0]$,
\begin{equation}\label{j8}
\begin{aligned}
I_{2,1}^{\delta}&=\log^4(1/\delta) \cdot\EE\Bigl[\int_0^1 \frac{1}{d(\eul_{\utn}, \Theta)^2}\cdot 1_{O_2^{\delta}}(\eul_{\utn})\cdot 1_{D_t^{\delta}}\,dt\Bigr]\\
&\leq \frac{9}{4}  \log^4(1/\delta)\cdot \EE\Bigl[\int_0^1\frac{1}{d(\eul_{t}, \Theta)^2}\cdot 1_{ \Theta^{\frac{3}{2}\varepsilon_1^{\delta}}\setminus \Theta^{\frac{1}{2}\varepsilon_2^{\delta}}}(\eul_{t})\,dt\Bigr].
\end{aligned}
\end{equation}
For $\delta\in(0, \delta_0]$ put $\overline\varepsilon^\delta=\delta^{3/4}\cdot \log^3(1/\delta)$ and observe that $\varepsilon_2^{\delta}\leq \overline\varepsilon^\delta\leq \varepsilon_1^{\delta}$ for all $\delta\in(0, \delta_0]$. Hence, \eqref{j8} implies that for all $\delta\in(0, \delta_0]$,
\begin{align*}
I_{2,1}^{\delta}
&\leq \frac{9}{4}  \log^4(1/\delta)\cdot \EE\Bigl[\int_0^1\frac{1}{d(\eul_{t}, \Theta)^2}\cdot 1_{ \Theta^{\frac{3}{2}\varepsilon_1^{\delta}}\setminus \Theta^{\overline\varepsilon^\delta}}(\eul_{t})\,dt\Bigr]+\\
&\qquad+ \frac{9}{4}  \log^4(1/\delta)\cdot \EE\Bigl[\int_0^1\frac{1}{d(\eul_{t}, \Theta)^2}\cdot 1_{ \Theta^{\overline\varepsilon^\delta}\setminus \Theta^{\frac{1}{2}\varepsilon_2^{\delta}}}(\eul_{t})\,dt\Bigr]\\
&\leq \frac{9}{4}  \log^4(1/\delta)\cdot \EE\Bigl[\int_0^1\frac{1}{\max(\overline\varepsilon^\delta,d(\eul_{t}, \Theta))^2}\cdot 1_{ \Theta^{\frac{3}{2}\varepsilon_1^{\delta}}}(\eul_{t})\,dt\Bigr]\\
&\qquad +\frac{9}{4}  \log^4(1/\delta)\cdot \EE\Bigl[\int_0^1\frac{1}{\max(\tfrac{1}{2}\varepsilon_2^{\delta},d(\eul_{t}, \Theta))^2}\cdot 1_{ \Theta^{\overline\varepsilon^\delta}}(\eul_{t})\,dt\Bigr].
\end{align*}
Applying Lemma \ref{occup} with $f=1/\max(\overline\varepsilon^\delta,\cdot)^2$ and $\gamma=1/2$ and with  $f=1/\max(\tfrac{1}{2}\varepsilon_2^{\delta},\cdot)^2$ and $\gamma=1/6$ we therefore conclude that there exist $c_1, c_2, c_3\in(0, \infty)$ such that for all $\delta\in(0, \delta_0]$,
\begin{equation}\label{vv2b}
\begin{aligned}
I_{2,1}^{\delta}
&\leq c_1\cdot \log^4(1/\delta)\cdot \Bigl(\int_0^{\frac{3}{2}\varepsilon_1^{\delta}}\frac{1}{\max(\overline\varepsilon^\delta,x)^2}dx+(\overline\varepsilon^\delta)^{-2}\cdot (\varepsilon_1^{\delta}+\delta)\\
&\qquad\qquad+\int_0^{\overline\varepsilon^\delta}\frac{1}{\max(\tfrac{1}{2}\varepsilon_2^{\delta},x)^2}dx+(\varepsilon_2^\delta)^{-2}\cdot \bigl((\overline\varepsilon^{\delta})^{\frac{4}{3}}+\delta^{\frac{4}{3}}\bigr)\Bigr)\\
&\leq c_2\cdot \log^4(1/\delta)\cdot\bigl((\overline\varepsilon^\delta)^{-1}+(\overline\varepsilon^\delta)^{-2}\cdot \varepsilon_1^{\delta}+(\varepsilon_2^\delta)^{-1}+(\varepsilon_2^\delta)^{-2}\cdot (\overline\varepsilon^{\delta})^{\frac{4}{3}}\bigr)\leq c_3\cdot \delta^{-1}.
\end{aligned}
\end{equation}
Moreover, employing \eqref{vv2} and Lemma \ref{dist}(ii) with $\alpha=1/2$ and $q=2$ we obtain that there exists $c\in(0, \infty)$ such that for all $\delta\in(0, \delta_0]$,
\begin{equation}\label{vv3}
\begin{aligned}
I_{2,2}^{\delta}&\leq \delta^{-2}\cdot \log^{-4}(1/\delta)\cdot \int_0^1 \PP\bigl(|\eul_{t}-\eul_{\utn}|> \tfrac{1}{2}d(\eul_{\utn}, \Theta),\, \eul_{\utn}\in \Theta^{\varepsilon_1^{\delta}}\setminus \Theta^{\varepsilon_2^{\delta}}\bigr)\,dt\\
&\leq c\cdot \log^{-4}(1/\delta).
\end{aligned}
\end{equation}
Combining \eqref{vv1}, \eqref{vv2b} and \eqref{vv3} we obtain \eqref{ww1}. This completes the proof of the estimate ~\eqref{ll32} in Theorem~\ref{Thm1}.

\section*{Acknowledgement}
I am
grateful to Thomas M\"uller-Gronbach for stimulating discussions
on the topic of this article.

\bibliographystyle{acm}
\bibliography{bibfile}

\end{document}